\documentclass[a4paper,english, 11pt]{article}

\usepackage[utf8]{inputenc}%caractères accentués
\usepackage[T1]{fontenc}
\usepackage{babel}
\usepackage[babel]{csquotes}

\usepackage{amsmath}%substack, displaystyle
\usepackage{amsfonts}%mathbb, mathfrak, mathcal
\usepackage{amsthm}%newtheorem, proof
\usepackage{bbm}%fonction indicatrice
\usepackage{amssymb}%sphericalangle
\usepackage[top=3.2cm, bottom=3.3cm, left=3.6 cm, right=3.6 cm]{geometry}
\usepackage{mathrsfs}
\usepackage[pdftex, pdfstartview=FitW,colorlinks=true,linkcolor=blue,%
urlcolor=blue,citecolor=blue]{hyperref}
\usepackage{stmaryrd}
\usepackage{cleveref}
\crefname{thm}{Theorem}{Theorems}
\crefname{fact}{Fact}{Facts}

\usepackage{float}%choisir l'emplacement des images
\usepackage{tikz}%dessins
\usetikzlibrary{shapes.misc}
\tikzset{cross/.style={cross out, draw, 
minimum size=2*(#1-\pgflinewidth), 
inner sep=0pt, outer sep=0pt}}
\usepackage{IEEEtrantools}%equations trop longues
\usepackage{url}%tilde dans les url

\usepackage{enumitem}%choisir la numération dans enumerate

\newtheorem{th.}{Theorem}[section]
\newtheorem{thm}[th.]{Theorem}

\newtheorem{fact}[th.]{Fact}
\newtheorem{lemma}[th.]{Lemma} 

\newtheorem*{question}{Question} 
\newtheorem{proposition}[th.]{Proposition}

\newtheorem{corollary}[th.]{Corollary}
\newtheorem*{rem.}{Remark}
\theoremstyle{definition}
\newtheorem{definition}[th.]{Definition}
\newtheorem{example}[th.]{Example}

\newcommand{\N}{\mathbb{N}}
\newcommand{\Z}{\mathbb{Z}}
\newcommand{\C}{\mathbb{C}}
\newcommand{\R}{\mathbb{R}}

\newcommand{\E}{\mathbb{E}}
\newcommand{\bP}{\mathbb{P}}

\newcommand{\eps}{\varepsilon}

\newcommand{\ks}{\mathfrak{s}}
\newcommand{\kk}{\mathfrak{k}}
\newcommand{\kr}{\mathfrak{r}}
\newcommand{\ka}{\mathfrak{a}}

\newcommand{\kg}{\mathfrak{g}}

\newcommand{\kn}{\mathfrak{n}}

\newcommand {\cW} {{\mathcal W}}

\newcommand {\tG} {\tilde G}

\newcommand {\tmu} {{\widetilde \mu}}
\newcommand {\tz} {{\widetilde z}}

\newcommand {\hz} {{\widehat z}}

\newcommand {\ui} {\underline{i}}

\newcommand {\ua} {\underline{a}}
\newcommand {\ux} {\underline{x}}
\newcommand{\acts}{\curvearrowright}

\newcommand{\dd}{\,\mathrm{d}}

\DeclareMathOperator{\SEP}{(SEP)}

\DeclareMathOperator{\supp}{supp}

\DeclareMathOperator{\End}{End}

\DeclareMathOperator{\SO}{SO}

\DeclareMathOperator{\Lie}{Lie}

\DeclareMathOperator{\Span}{span}

\DeclareMathOperator{\Ker}{Ker}

\DeclareMathOperator{\id}{id}
\DeclareMathOperator{\ad}{ad}
\DeclareMathOperator{\Ad}{Ad}

\DeclareMathOperator{\leb}{leb}
\DeclareMathOperator{\Heis}{{\mathcal H}}
\DeclareMathOperator{\Haar}{Haar}
\DeclareMathOperator{\Acc}{Acc}

\title{Classification of Choquet-Deny Lie groups}
\author{Timoth\'ee B\'enard\thanks{CNRS -- LAGA, Universit\'e Sorbonne Paris Nord, 99 avenue J.-B. Cl\'ement, 93430 Villetaneuse, France.  E-mail: \href{mailto:timothee.benard@gmail.com}{timothee.benard@gmail.com}.}\  \,and Joshua Frisch\thanks{Department of Mathematics, University of California San Diego, 9500 Gilman Drive, La Jolla, CA 92093, United States. E-mail: \href{mailto:joshfrisch@gmail.com}{joshfrisch@gmail.com}.}}
\date{September 2026}

\begin{document}
	
\large

\maketitle

\begingroup
\renewcommand{\thefootnote}{}
\addtocounter{footnote}{-1}
\endgroup

\begin{abstract}
We characterize connected Lie groups on which all adapted random walks  have only trivial bounded harmonic functions. In the nilpotent case, this answers a question of Furstenberg recorded by Guivarc'h in 1973.
\end{abstract}

\tableofcontents

\section{Introduction}

Let $G$ be a real Lie group and let $\mu$ be a probability measure on
$G$. A bounded measurable function $f:G\to\R$ is called
\emph{$\mu$-harmonic} if
\[
 f(x)=\int_G f(xg)\,\dd\mu(g)
\]
for almost every $x\in G$. The pair $(G,\mu)$ is called
\emph{Liouville} if every bounded $\mu$-harmonic function on $G$ is almost
everywhere constant. We call $\mu$ \emph{adapted} if its support
generates a dense subgroup of $G$, and \emph{spread-out} if some
convolution power $\mu^{*k}$ is not singular with respect to the Haar
measure, or equivalently has a nonzero absolutely continuous part.
The group $G$ is \emph{Choquet-Deny} if $(G,\mu)$ is Liouville for every
adapted probability measure $\mu$. We use
\emph{Choquet-Deny-spread-out} for the corresponding property restricted
to adapted spread-out measures.

It has been known since at least the 1940s that compact Lie groups are Choquet-Deny (for example,  it follows from  It\^o-Kawada \cite{KawadaIto1940}). The fact that $\Z^d$ is Choquet-Deny is due to Blackwell \cite{Bla55} in 1955. 
His result was extended to all abelian locally compact groups by Choquet and Deny \cite{ChoquetDeny1960} in 1960. The property is named after them. 

For finitely generated non-abelian groups, Dynkin and Malyutov \cite{DM61} subsequently proved that virtually nilpotent groups are Choquet-Deny.  The converse was  established by Erschler  in restriction to solvable groups \cite{Erschler2004}, then in full generality by the second author, Hartman, Tamuz, and Vahidi Ferdowsi \cite{FHTVF}. More generally, for all discrete countable groups, the Choquet-Deny property has been characterized  as the absence of ICC quotients, see respectively   Jaworski \cite{Jaw04} and \cite{FHTVF} for the ``if'' and ``only if''  directions of the equivalence.

In connected settings, Azencott \cite{AzencottLNM1970,Cartier1971} classified in 1970 all connected Lie groups that are Choquet-Deny-spread-out. Exploiting Furstenberg's boundary theory \cite{Furstenberg63}, he showed they are exactly the connected Lie groups $G$ whose radical $R$ satisfies that $G/R$ is compact and where every adjoint action $\Ad(r)\in \Ad(R)\acts \text{Lie}(R)$ has all its eigenvalues of modulus $1$. Soon after, Guivarc'h \cite{Guivarch73b} showed that Choquet-Deny-spread-out connected Lie groups are exactly those of polynomial growth. His   characterization was subsequently extended by Jaworski to connected second countable locally compact groups, not necessarily of Lie type \cite[Theorem~3.16]{Jaworski95}.

%For symmetric measures, the class of groups with trivial Poisson boundary is larger: Birg'e and Raugi proved that every symmetric, adapted, spread-out probability measure on a connected amenable Lie group has trivial Poisson boundary \cite{BR74}.

%For discrete groups, the Choquet--Deny property goes back to Blackwell \cite{Bla55}, who established it for \(\mathbb{Z}^d\). Dynkin and Malyutov \cite{DM61} subsequently proved it for finitely generated virtually nilpotent groups. More generally, Jaworski \cite{Jaw04} proved that every FC-hypercentral countable group (those with no ICC quotient) is Choquet--Deny. The second author, Hartman, Tamuz, and Vahidi Ferdowsi \cite{FHTVF19} proved the converse. Thus, a countable group is Choquet--Deny if and only if it has no ICC quotient; in particular, a finitely generated group is Choquet--Deny if and only if it is virtually nilpotent.

Beyond compact and abelian groups, all situations described above refer to spread-out measures. Classifying Choquet-Deny Lie groups in the \emph{non-spread-out case} remained mysterious, even in the nilpotent setting. As reported by Guivarc'h  \cite{Guivarch73b} in 1973, the following question was asked by Furstenberg.
\begin{question}[Furstenberg, Guivarc'h]
Is every nilpotent Lie  group Choquet-Deny? 
\end{question}

In  \cite{Guivarch73b}, Guivarc'h  answered the question positively for nilpotent groups of step at most $2$, or arbitrary step under an additional positive moment assumption on the driving measure $\mu$. Raugi \cite{Raugi04} later claimed to establish the unrestricted
result, i.e. that every locally compact second countable nilpotent group is Choquet-Deny. However, Raugi's paper  carries a fatal mistake\footnote{Lemma~2.5 in Raugi \cite{Raugi04} is false, as noted by \cite[p.~160]{JR07}. },   invalidating the proof. 
In the present paper we provide an answer to this question of Furstenberg and Guivarc'h.
Against all expectations,  contrary to the finitely generated/spread-out situations, and contrary to Raugi's prediction which was supported by Guivarc'h's work in step $2$, we show that \emph{most nilpotent connected Lie groups are  not Choquet-Deny}!  In fact, we characterize exactly when this is the case in terms of the lower central series of the group.

\begin{thm}[Choquet-Deny nilpotent Lie groups]\label{CD-Nilpotent}
Let $G$ be a connected nilpotent real Lie group. Then $G$ is
Choquet-Deny if and only if the iterated commutator group
$[G,[G,G]]$ is bounded in $G$.
\end{thm}

Here, we use \emph{bounded} as a synonym for relatively compact. As a consequence\footnote{Indeed, every bounded  subgroup of a nilpotent  Lie group is central
(\Cref{max-comp-in-nilp}).} of \Cref{CD-Nilpotent}, every connected nilpotent Lie group  is Choquet-Deny if it has step at most $2$, and is non-Choquet-Deny if it has step at least 4. In step $3$, both behaviors occur depending on the group, as we see in the next example.

\begin{example}
Let $\tG$ be a simply connected nilpotent Lie group of step $3$, and
let $\Lambda$ be a lattice in the  group
$\tG^{[3]}=[\tG,[\tG,\tG]]$ (which we may see as a vector space). Then $G=\tG/\Lambda$ is Choquet-Deny,
because $G^{[3]}=\tG^{[3]}/\Lambda$ is compact, whereas its universal
cover $\tG$ is not Choquet-Deny.
\end{example}

This contrasts with the case of finitely generated  groups, for which being Choquet-Deny amounts to being virtually nilpotent \cite{DM61, FHTVF}. It also differs from Azencott's result establishing that
connected nilpotent Lie groups are always Choquet-Deny-spread-out \cite{AzencottLNM1970}. 
%It also shows that singular measures are not merely a technical complication: they can detect global topological information invisible to spread-out measures.

\bigskip
Our main result reaches beyond the nilpotent setting and gives a simple
characterization of the Choquet-Deny property for \emph{every} connected real
Lie group. Moreover, failure to be Choquet-Deny can  always be witnessed by a symmetric adapted measure.

\begin{thm}[Choquet-Deny Lie groups]\label{main-thm}
Let $G$ be a connected real Lie group. The following conditions are equivalent:
\begin{enumerate}[label=(\roman*)]
\item $G$ is Choquet-Deny;
\item every  $x\in [G,G]$ has bounded conjugacy class
$\{gxg^{-1}:g\in G\}$.
\end{enumerate}
Moreover, if $(i)$ and $(ii)$ fail, then $G$ admits a \emph{symmetric} non-Liouville adapted probability measure on $G$.\end{thm}

Notably,  \Cref{main-thm} implies that being Choquet-Deny is stable  under passage to connected  closed subgroups.

It is interesting to compare the final claim about symmetry with  the spread-out theory. Indeed, a connected Lie group is Choquet-Deny in restriction to symmetric spread-out measures if and only if it is amenable \cite{AzencottLNM1970, BR74}. 
Therefore,  any amenable connected Lie group of exponential growth fails to be Choquet-Deny-spread-out but this cannot be witnessed by exhibiting a symmetric non-Liouville spread-out measure. That we may further  impose  symmetry in \Cref{main-thm}  resonates with the discrete setting \cite{FHTVF}.

%For symmetric measures, the class of groups with trivial Poisson boundary is larger: Birg'e and Raugi proved that every symmetric, adapted, spread-out probability measure on a connected amenable Lie group has trivial Poisson boundary \cite{BR74}.
\bigskip

When the structure of $G$ is specified, the criterion from \Cref{main-thm} (ii) can take other forms. We record two useful instances.

\begin{corollary}[Solvable case]\label{solv-CD-crit}
Assume that $G$ is solvable, and set $S=\overline{[G,G]}$. Then $G$ is
Choquet-Deny if and only if the conjugation action of $G$ on the
abelianization $S/\overline{[S,S]}$ has bounded orbits and
$[S,S]$ is bounded.
\end{corollary}

\begin{proof}Consequence of \Cref{main-thm} and \Cref{solv-SEP}. 
\end{proof}

\begin{example} \label{ex-SO2Heis}
Let $\Heis=(\R^2\times \R,*)$ be the Heisenberg group, i.e. the set $\R^2\times \R$ equipped with the  group structure 
$(v,s)*(w,t)=(v+w, \,s+t+v\wedge w)$ where $\wedge$ stands for the determinant $v\wedge w=v_{1}w_{2}-v_{2}w_{1}$. Let $\SO_2$ act on $\Heis$ by automorphisms via the formula
$k.(v,s)=(kv,s)$. Set $ e_3=(0,0,1)\in \Heis$, and
\[
 G_{1}=\SO_2\ltimes\Heis,
 \qquad
 G_{2}=\SO_2\ltimes(\Heis/\Z e_3).
\]
Then $G_{2}$ is Choquet-Deny whereas $G_{1}$ is not. Moreover,
$[G_{2},[G_{2},G_{2}]]=[G_{2},G_{2}]=\Heis/\Z e_3$ is unbounded. 
\end{example}

This example illustrates that the bounded three-commutator criterion, characterizing  the Choquet-Deny property  for nilpotent groups, is not valid for solvable groups. In passing, it also shows that a compact extension of a Choquet-Deny Lie group (here $G_{1}$ extending $\Heis$) can fail to be Choquet-Deny.  In contrast,  \Cref{main-thm} implies that being Choquet-Deny is insensitive to quotienting by a normal compact subgroup $K \unlhd G$, see also Jaworski-Raja \cite[Corollary~4.4]{JR07} for a related statement under a distality assumption.

\begin{corollary}[Nilpotent-by-compact case]\label{cpct-by-nilp-cor}
Let $G=K\ltimes N$, where $K$ is a connected compact Lie group and
$N$ is a connected nilpotent Lie group. Then $G$ is Choquet-Deny if
and only if the three subgroups
\[
 [[K,K],N],\qquad [[K,N],N],\qquad [[N,N],N]
\]
are bounded.
\end{corollary}

The proof of \Cref{cpct-by-nilp-cor} is given at the end of \Cref{Sec-onlyif}.

\bigskip

\noindent{\bf Further related work.} 
%Although they are not formulated in the Choquet-Deny framework, other prior works have studied triviality of bounded harmonic functions on groups (equivalently triviality of Poisson boundary). Let us mention a few.
%As previously mentioned, for finitely generated d iscrete groups, Dynkin and Malyutov showed that virtually nilpotent groups are Choquet-Deny
%\cite{DynkinMalyutov1961}. Erschler established the converse among
%finitely generated solvable groups \cite{Erschler2004}, and Frisch,
%Hartman, Tamuz, and Vahidi Ferdowsi obtained the complete classification:
%a finitely generated group is Choquet-Deny if and only if it is virtually
%nilpotent \cite{FHTVF}. More generally, their criterion for a countable
%discrete group is the absence of an ICC quotient. 
Furstenberg showed that every full-support probability measure on a nonamenable group has a nontrivial Poisson boundary \cite{Fur73} and conjectured the converse. 
This conjecture was proved independently by Rosenblatt
\cite{Ros81} and Kaimanovich and Vershik
\cite{KV83} for locally compact
second-countable groups: every amenable such group admits
a full-support probability measure with trivial Poisson boundary.
Moreover, the measure can be chosen symmetric.

Restricting  $\mu$  to have finite support, Avez \cite{Avez74} showed that every adapted finitely supported measure on a discrete group of subexponential growth is Liouville. Kaimanovich and Vershik \cite{KV83} found solvable groups of exponential growth where every adapted symmetric finitely supported measure is Liouville and  other solvable groups where every adapted symmetric finitely supported measure is non-Liouville. Bartholdi and Erschler \cite{bartholdi2017poisson} found an example of a group of exponential growth where all finitely supported measures have a trivial Poisson boundary. 

Finally, for totally disconnected  groups, the Choquet-Deny property has instead
been related by Raja \cite{Raja05} and Jaworski-Raja  \cite{JR07} to polynomial growth, distal inner automorphisms,
contraction groups, and the SIN property.

\bigskip

\noindent{\bf Proof strategy.} The two implications in \Cref{main-thm}
require rather different
ideas. 

The sufficient direction $(ii) \implies (i)$  builds on the  method of periods
of harmonic functions, developed notably by Azencott
\cite[Chapter~IV]{AzencottLNM1970}, Guivarc'h
\cite[Chapter~V]{Guivarch73b}, and Raugi
\cite[Section~6]{Raugi1977}, \cite{Raugi1978}.
A period is a group element under whose translation a  function
is invariant. Identifying periods of bounded harmonic
functions allows one to factor them through suitable quotients and thus reduce the complexity of the group at hand. In our proof, we show accumulation points of
randomly conjugated commutators yield right periods of martingale
limit functions. We then argue by induction to show the identified periods do not depend on the limit function, and thus yield  periods of the original
harmonic function. Applied iteratively, this
establishes  $(ii) \implies (i)$ in \Cref{main-thm}.

%For the sufficient direction, bounded conjugacy classes produce additional periods of bounded harmonic functions; an induction on the dimension of the group, combined with Lie structure theory, then forces every such function to be constant. 

For the direction $(i) \implies (ii)$, we argue by contrapositive, assuming $(ii)$ fails, i.e. some 
commutator has an unbounded conjugacy class. We promote this to a generic 
\emph{simultaneous escape property}: the multiplicative differences
between many prescribed words in generic enough elements can all be sent to infinity or to the identity by applying two-sided multiplication by $\sigma^{\pm1}_{k}$ where $(\sigma_{k})_{k\geq0}$ is a well-chosen sequence in $G$. We use it to construct a
symmetric fractal probability measure whose support generates a dense
subgroup but whose induced random walk on $G$  retains a macroscopic memory of its
starting point. Foguel's criterion \cite{Foguel75} then gives a
nonconstant bounded harmonic function.  Note the  argument does not identify the
Poisson boundary or exhibit an explicit non-trivial harmonic function. Our proof also resonates with \cite{FHTVF}, although the latter is about the discrete case in which  conclusions are very different. 
\bigskip

\noindent{\bf Open question.}
 Is there a common conjugacy-dynamical criterion for locally compact
groups that recovers the absence of ICC quotients in the countable
discrete setting, the distal/SIN criteria in the totally disconnected
setting, and the bounded-conjugacy-class condition of \Cref{main-thm}?

\bigskip

\noindent\textbf{Structure of the paper.}
\Cref{Prel-Lie groups} collects standard preliminaries on Lie groups and
harmonic functions. \Cref{if-mainthm} proves the ``if'' direction of
\Cref{main-thm}. \Cref{SEP-section} prepares the ``only if'' direction by
introducing and establishing a simultaneous escape property for the
conjugation action of $G$ on suitable products; the nilpotent, solvable,
and semisimple cases are treated separately. This escape property includes signs to accommodate
symmetric random walks. Finally, \Cref{Sec-onlyif}
uses this property to construct a symmetric non-Liouville adapted measure
whenever a commutator has unbounded conjugacy class.

To make the nilpotent case especially accessible, we have arranged the exposition so that it can be read through a  streamlined route, avoiding the heavier arguments necessary for the general Lie group case. More precisely, readers interested only in \Cref{CD-Nilpotent} may start reading \Cref{Lie groups} from \Cref{fact-nilp-spc}, stop reading \Cref{if-mainthm} at \Cref{CD-nilp-setting}, and in \Cref{SEP-section} need only read \S\ref{nilpSEP}.

\bigskip

\noindent\textbf{Notation.}
All Lie groups in this paper are real. For $g\in G$, we write $\Ad(g)$
both for the conjugation map $x\mapsto gxg^{-1}$ and, when no confusion
can arise, for its differential on the Lie algebra $\kg$. Brackets
$[\cdot,\cdot]$ denote group commutators, while
$[\cdot,\cdot]_{\kg}$ denotes the Lie bracket on $\kg$.

%%%% Write Ad(g)x or Ad(g)(x)?
\bigskip
\noindent\textbf{AI disclosure.} 
The main ideas and constructions of this paper, including the treatment of nonsymmetric measures, were developed without the use of any AI tool, in work begun in early 2025. 
At a later stage, while adapting the construction to produce \emph{symmetric} non-Liouville measures (i.e. establishing the ``moreover'' part of \Cref{main-thm}), we  used ChatGPT as a discussion tool. All mathematical arguments and the final text were checked and written by the authors.

\bigskip
\noindent\textbf{Acknowledgements.} 
We thank Vadim Kaimanovich for his careful reading, valuable
comments, and guidance concerning the historical development
of the subject.

\section{Preliminaries} \label{Prel-Lie groups}

We record standard properties regarding the structure of Lie groups (\S\ref{Lie groups}) and the regularity of harmonic functions (\S\ref{Harmonic functions}).

\subsection{Lie groups} \label{Lie groups}

Let $G$ be a connected  Lie group. 

\begin{fact}[Real-analytic structure] \label{analyt-struct}
 $G$ admits a structure of real-analytic manifold for which the product map and the inversion map are real-analytic. Moreover, this structure is unique, and  exponential local charts are real-analytic. 
\end{fact}

\begin{proof}
See \cite[Proposition~1.117]{Knapp2002}.
\end{proof}

We will always consider $G$ endowed with this structure of real-analytic manifold. In view of Cartan's theorem, every closed subgroup of $G$ is an embedded real-analytic Lie subgroup (and conversely). 

\bigskip

 Let $\kg$ be the Lie algebra of $G$. Recall $\kg$ is  \emph{simple} if it is non-abelian and its only ideals are $\{0\}$ and $\kg$.  More generally, $\kg$ is \emph{semisimple}  if it is the direct sum of simple Lie algebras, or equivalently if $\{0\}$ is the only abelian ideal of $\kg$. We say that $G$ is simple or semisimple if this is the case of its Lie algebra.
  
 \begin{fact}[Weyl's theorem] \label{Weyl} If $G$ is semisimple and compact, then the fundamental group of $G$ is finite.
\end{fact}
 \begin{proof} See \cite[Theorem~4.26]{Knapp1986}.
\end{proof}

% Denote by $(^{[k]}{G})_{k\geq1}$ and $(^{[k]}{\kg})_{k\geq1}$ the derived series of $G$ and $\kg$ respectively. Namely, $(^{[k]}{G})_{k\geq1}$ is  defined by induction via the rules $^{[1]}{G}=G$ and $^{[k+1]}{G}=[^{[k]}{G}, {^{[k]}{G}}]$, and $(^{[k]}{\kg})$ is defined similarly using the Lie bracket $[.,.]_{\kg}$ instead of the group commutator. Note $^{[k]}{\kg}$ is the Lie algebra of $^{[k]}{G}$ (which is an immersed subgroup). 
%We say that $G$ (equivalently $\kg$) is solvable if the derived series $(^{[k]}{G})_{k\geq1}$  (equivalently  $(^{[k]}{\kg})_{k\geq1}$) reaches $\{\id\}$ (resp. $\{0\}$)for some finite $k$.

 Denote by $(^{[k]}{\kg})_{k\geq1}$  the derived series of $\kg$. It is the decreasing sequence of ideals defined by the rules $^{[1]}{\kg}=\kg$ and $^{[k+1]}{\kg}=[^{[k]}{\kg}, {^{[k]}{\kg}}]_{\kg}$ where $[.,.]_{\kg}$ denotes the Lie bracket.
The Lie algebra $\kg$ is  \emph{solvable} if $^{[k]}{\kg}=\{0\}$ for some finite $k\geq 1$.  We say that $G$ is solvable if it is the case of $\kg$.  Note the derived series  $(^{[k]}{G})_{k\geq1}$ of $G$,  defined by $^{[1]}{G}=G$ and $^{[k+1]}{G}=[^{[k]}{G}, {^{[k]}{G}}]$, satisfies that each subgroup $^{[k]}{G}$ is immersed in $G$, with Lie algebra $^{[k]}{\kg}$. Hence the property of being solvable can also be read in the derived series of $G$.
%$\{\id\}$ is the only connected normal abelian subgroup of $G$. 

\begin{fact}[Radical] \label{def-rad}
$G$ admits a largest solvable  normal connected subgroup $R$. It is closed and the quotient Lie group $G/R$ is semisimple. 
\end{fact}

$R$ is called the \emph{radical} of $G$.
\begin{proof}
See \cite[\S 1.4]{GOV94}.
\end{proof}
%Let us check that $N$ is well defined. Passing to the closure, one may reduce to considering closed normal nilpotent connected subgroups. By considering dimension, it is enough to show that this class is stable by multiplication, i.e. by $(R,S)\mapsto RS$. To check this stability, one notes  that for any normal subgroups $R,S,T$, one has the relation $[RS, T]\subseteq [R,T][S,T]$. This implies $[RS, RS]\subseteq [R,R][R,S] [S,R] $, and more generally that $(RS)^{[k]}$ is included in the product of $k$-iterated commutators involving either $R$ or $S$. Using that  $[R, S^{[i]}]\subseteq S^{[i]}$, we deduce $(RS)^{[k]}$ is trivial for $k$ bigger that $\text{step}(R)+ \text{step}(S)$.

The next result ensures that $\kg$ is the semidirect product of its largest solvable ideal and a semisimple complementary Lie subalgebra.

\begin{fact}[Levi decomposition] \label{Levi-dec} Let $\kr\subseteq \kg$ be the Lie algebra of $R$. There exists  a semisimple Lie subalgebra  $\kk\subseteq \kg$ which is in direct sum with $\kr$ (for the vector space structure of $\kg$). Moreover, any two such complementary semisimple Lie subalgebras
are conjugate under $\Ad(G)$.
\end{fact}

\begin{proof}
Existence follows from the Levi decomposition
\cite[Appendix~B, Theorem~B.2]{Knapp2002}.
The conjugacy assertion is the Levi--Malcev theorem; see
\cite[Theorem~1.4.3]{GOV94}.
\end{proof}

The immersed connected subgroup of $G$ corresponding to $\kk$ might not be closed. It is at least closed when $G/R$ is compact.  

\begin{corollary} \label{Levi-G/R-compact}
If $G/R$ is compact, then $G= K R$ where $K$ is a semisimple compact connected Lie subgroup with $K\cap R$ finite. Moreover, $K$ is unique up to conjugation by $G$.
\end{corollary}

\begin{proof}
Let $\kk$ be as in  \Cref{Levi-dec}, let $\iota : K\hookrightarrow G$ be the \emph{immersed} connected Lie group corresponding to $\kk$. Set $\pi :G\rightarrow G/R$ the projection map. Then $\pi \circ \iota : K\rightarrow G/R$ is a surjective Lie group homomorphism with discrete kernel, therefore it is a covering. 
As $G/R$ is compact and semisimple, it has finite fundamental group (\Cref{Weyl}), which implies the fibers  of $\pi \circ \iota$ are finite. It follows that $K$ is compact, and $\iota$ is an embedding. From now on, we identify $K$ and $\iota(K)$, so we see $K\subseteq G$, and it validates the first claim.
We now check the uniqueness statement. If $K'$ is another valid option, then its Lie algebra $\kk'$ is also semisimple, and in direct sum with $\kr$ because $K'\cap R$ is finite. Applying \Cref{Levi-dec}, we may consider $g\in G$ such that $\Ad(g)\kk'=\kk$, and by connectedness it follows that $gK'g^{-1}=K$. 
\end{proof}

The next fact will be useful to ensure that a connected Lie group acts by isometries on each of its compact normal subgroups.

\begin{fact}\label{autom-K}
The group of  automorphisms of a compact connected Lie group $K$ is a Lie group, whose connected component coincides with the subgroup of inner automorphisms.  In particular, it is discrete if $K$ is abelian. 
\end{fact}

\begin{proof}
See \cite{Hochschild1952} and \cite[Theorem~1, p.~509]{Iwasawa1949}.
\end{proof}

We  now turn to properties of nilpotent Lie groups. 
Write $(\kg^{[k]})_{k\geq 1}$  the lower central series of $\kg$. It is defined by induction via the rules $\kg^{[1]}=\kg$ and $\kg^{[k+1]}=[\kg, \kg^{[k]}]_{\kg}$. 
The Lie algebra $\kg$ is  \emph{nilpotent} if $\kg^{[k]}=\{0\}$ for some finite $k\geq 1$. We say that $G$ is nilpotent if it is the case of $\kg$.  Note the lower central series  $({G}^{[k]})_{k\geq1}$ of $G$,  defined by ${G}^{[1]}=G$ and ${G}^{[k+1]}=[{G}, G^{[k]}]$, satisfies that each subgroup $G^{[k]}$ is immersed in $G$, with Lie algebra ${\kg}^{[k]}$. Hence the property of being nilpotent can also be read in the lower central series of $G$.
We call \emph{step} of $\kg$ (or $G$)  the largest integer $s\geq1$ such that $\kg^{[s]} \neq \{0\}$ (equivalently $G^{[s]}\neq \{\id\}$).

\begin{fact} \label{fact-nilp-spc}
If $G$ is  nilpotent and simply connected, then the exponential map $\exp: \kg \rightarrow G$ is a diffeomorphism. The group structure $*$ induced by $G$ on $\kg$  is given by the Baker-Campbell-Hausdorff formula
\begin{equation*}
x*y=x+y+\frac{1}{2}[x,y]_{\kg}+ \,\cdots \tag{$x,y\in \kg$}
\end{equation*}
In particular, $(\kg, *)$ is an algebraic group (i.e. the product map and the inversion map are polynomial), and the exponential map induces a bijection between Lie subalgebras of $\kg$ and connected Lie subgroups of $G$.
\end{fact}

  \Cref{fact-nilp-spc} thus guarantees  that every simply connected nilpotent Lie group carries  an additional structure of real vector space by identification with its Lie algebra, and it makes it an algebraic group.
  
  \begin{proof}
\Cref{fact-nilp-spc} follows from \cite[Theorem~1.127, Corollary~1.134 and Theorem~B.22]{Knapp2002}.
\end{proof}

\begin{lemma} \label{max-comp-in-nilp} If $G$ is nilpotent, then $G$ admits a largest compact subgroup $T_{G}$. It is central, connected, and $G/T_{G}$ is simply connected.
\end{lemma}

\begin{proof} 
Consider a universal cover $\pi:\tG\rightarrow G$, set $\tilde H$ the Zariski-closure of $\Ker \pi$, and  $H:=\pi (\tilde H)$.  Note $\Ker \pi$ is a normal discrete subgroup in $\tG$. It is therefore central, so  $\tilde H$ is as well  by Zariski-density of $\Ker \pi\subseteq \tilde H$.  This means the group structure on $\tilde H$ is that of the underlying vector space (see \Cref{fact-nilp-spc}), and $\Ker \pi$ is an additive lattice in $\tilde H$. It follows that $H$ is central, compact, connected, and the quotient $G/H$ is a simply connected nilpotent Lie group (because isomorphic to $\tG/\tilde H$). It remains to check that $H$ contains all compact  subgroups of $G$.  Let $L\subseteq G$ be one of them. The projection of $L$ to $G/H$ is a compact subgroup in a simply connected nilpotent Lie group. It must therefore be trivial (as for any non-trivial $x\in G/H$, the map $\N\rightarrow G/H, k\mapsto x^k$ is proper by \Cref{fact-nilp-spc}), i.e. $L\subseteq H$. 
 This validates the lemma with  $T_{G}:=H$.
\end{proof}

Similarly to \Cref{def-rad}, we have 
\begin{fact}[Nilpotent radical]
Every connected Lie group $G$ admits a largest nilpotent normal  connected subgroup $N$. 
\end{fact}
We call $N$ the \emph{nilpotent radical} of $G$.

\begin{proof} 
See \cite[Chapter~2, \S5.2]{GOV94}.
\end{proof}

%Let us check that $N$ is well defined. Passing to the closure, one may reduce to considering closed normal nilpotent connected subgroups. By considering dimension, it is enough to show that this class is stable by multiplication, i.e. by $(R,S)\mapsto RS$. To check this stability, one notes  that for any normal subgroups $R,S,T$, one has the relation $[RS, T]\subseteq [R,T][S,T]$. This implies $[RS, RS]\subseteq [R,R][R,S] [S,R] [S,R]$, and more generally that $(RS)^{[k]}$ is included in the product of $k$-iterated commutators involving either $R$ or $S$. Using that  $[R, S^{[i]}]\subseteq S^{[i]}$, we deduce $(RS)^{[k]}$ is trivial for $k$ bigger that $\text{step}(R)+ \text{step}(S)$.

\bigskip
We finish the section with a useful concept, that of \emph{largest vector space quotient $G/G'$ of a Lie group $G$}. It will enable us to use tools from linear algebra even though we deal with abstract Lie groups. 

\begin{definition}[Largest vector space quotient] \label{vs-radical}
Given a connected Lie group $G$,  we define $G'$  as the smallest  normal closed subgroup of $G$ such that $G/G'$ is isomorphic to a real vector space (endowed with the addition).

 In more concrete terms, $G'$ is obtained as follows: $G'=\pi^{-1}(T)$, where $\pi:G\rightarrow G/\overline{[G,G]}$ is the quotient map, $\overline{[G,G]}$ is the closure of the derived subgroup, and $T\subseteq G/\overline{[G,G]}$ is the maximal compact subgroup (torus component). Note that $G'$ is stable under automorphisms of $G$, so every automorphism of $G$ acts on $G/G'$ as a linear map. 
\end{definition}

%To be added:  $[K,K]$ is closed in $K$ if $K$ is a compact Lie group? And counterexample without compactness. 

\subsection{Harmonic functions} \label{Harmonic functions}

Let $G$ be a Lie group, let $\mu$ be a probability measure on $G$. Denote by $P_{\mu}$  the operator on bounded measurable functions $f:G\rightarrow \R$ defined by $P_{\mu}f(x)=\int_{G}f(xg)d\mu(g)$. Note that $P_{\mu}$ preserves the set of almost-everywhere zero functions, therefore induces a linear operator $P_{\mu}\acts L^\infty(G)$. As mentioned earlier, a function $f:G\rightarrow \R$ is $\mu$-\emph{harmonic} if its class $[f]\in L^\infty(G)$ is fixed under $P_{\mu}$, i.e. $P_{\mu}[f]=[f]$.

Recall a function $f:G\rightarrow \R$ is \emph{left uniformly continuous} if for every $\eps>0$, there exists $U\subseteq G$ a neighborhood of the identity such that 
$$ \sup_{x\in G, u\in U}|f(x)- f(ux)|<\eps.$$

We now see that studying bounded harmonic functions on $G$ essentially boils down to studying those that are left uniformly continuous. This is a standard regularization argument; see
\cite[Lemma~1.3 and Corollary~1.5]{Raugi1978}.

\begin{lemma}
 Let $f:G\rightarrow \R$ be a bounded $\mu$-harmonic  function. There exists a sequence of  \emph{left uniformly continuous} bounded $\mu$-harmonic  functions $(f_{n}:G\rightarrow \R)_{n\geq 1}$ such that $f_{n}\to_{n} f$ almost-everywhere.
\end{lemma}

\begin{proof}
Note  the definition of harmonicity involves the \emph{right} $\mu$-random walk. Using left-convolution, we can always approximate a bounded harmonic function by a left uniformly continuous harmonic function and the result follows. 
 \end{proof}

\begin{corollary}
The pair $(G,\mu)$ is Liouville if and only if all left uniformly continuous  $\mu$-harmonic bounded functions are constant. 
\end{corollary}

\section{Sufficient condition to be Choquet-Deny}\label{if-mainthm}

In this section, we show the ``if'' part in \Cref{main-thm}: 

\begin{proposition} \label{CD-sufficient-Liegrp}
Let $G$ be a  connected Lie group such that for every $x\in [G,G]$, we have $\Ad(G)x$ bounded. Then $G$ is Choquet-Deny.
\end{proposition}

\bigskip
We start by recalling two well known cases. 

\begin{lemma}[Compact case] \label{CD-compact-case}
Let $G$ be a  compact  Lie group. Then $G$ is Choquet-Deny.
\end{lemma}

\begin{proof}
Let $\mu$ be a probability measure on $G$ whose support generates a dense subgroup.  Let $f :G\rightarrow \R$ be a continuous bounded harmonic function. Set $M=\max_{G}f$. By harmonicity and continuity of $f$, the level set $f^{-1}(M)$ is $\supp \mu$-invariant, more precisely: for every $x\in f^{-1}(M)$, for every $g\in \supp \mu$, one has $xg \in f^{-1}(M)$. Set $\Gamma_{\mu}^+$ the semigroup generated by $\supp \mu$. Its closure $\overline{\Gamma_{\mu}^+}$ is a compact semigroup, therefore a group, so $\Gamma_{\mu}^+$ is dense in $G$ as well. This yields that $f^{-1}(M)$ is dense in $G$, so $f$ is constant by continuity. 

\end{proof}

\begin{lemma}[Abelian case] \label{CD-abelian-case}
Let $G$ be an abelian  Lie group. Then $G$ is Choquet-Deny.
\end{lemma}

This result is established in  \cite{ChoquetDeny1960}. We give an alternative proof extracted from \cite{Raugi04}. It yields insight on the general case. 

\bigskip
Let $G$ be a Lie group, $\mu$ a probability measure on $G$. Set $(B, \beta)=(G^{\N^*}, \mu^{\otimes \N^*})$ the space of instructions. Let $f : G \rightarrow \R$ be a  bounded harmonic   function on $G$ which is left uniformly continuous. Given $x\in G$, the sequence of functions $\varphi_{n} : b\mapsto f(xb_{1}\dots b_{n})$ is a bounded martingale on $(B, \beta)$, whence it converges $\beta$-almost surely. Using a countable dense family of points $x$ in $G$ and the left uniform continuity of $f$, we deduce there exists a measurable collection of left uniformly continuous bounded functions $(f_{b}:G\rightarrow \R)_{b\in B}$ such that: 
  \begin{align} \label{decomp-f}
f=\int_{B}\,f_{b} \dd\beta(b)
\end{align}
and for $\beta$-almost every $b\in B$, for all  $x\in G$, 
\begin{align} \label{def-fb}
f_{b}(x)=\lim_{n} f(x b_{1}\dots b_{n}). 
\end{align}
We call $(f_{b})_{b\in B}$ the \emph{limit functions} associated to $f$. The next lemma strengthens \eqref{def-fb} by allowing us to insert an extra factor $g$ after $b_{n}$.  We denote by $\mu^{*k}$ the $k$-fold convolution power of $\mu$.

\begin{lemma}\label{limit-decomp}
For $k\geq 1$, for $\beta \otimes \mu^{*k}$-almost every $(b,g)\in G^{\N^*}\times G$, for all $x\in G$, we have 
\begin{align} \label{eqfbg}
f_{b}(x)=\lim_{n} f(x b_{1}\dots b_{n} g) 
\end{align}
\end{lemma}

The result already appears in Raugi  \cite[Lemma~1.7 and Corollaries~1.8 and~1.10]{Raugi1978}, and in fact, the square-increment argument underlying the proof precedes it in Cartier's exposition of Azencott's
work \cite[proof of Theorem~3, p.~121]{Cartier1971}. We include the short proof for completeness.

\begin{proof}
One may suppose $x$ fixed (using left uniform continuity of $f$). For $b\in B$, $g\in G$, set
 $$S(b,g):=\sum_{n\geq 0} (f(x b_{1}\dots b_{n} g) -f(x b_{1}\dots b_{n} ) )^2.$$
Using that $f$ is $\mu$-harmonic, we can upper bound the integral of $S$:
$$\int_{B\times G} S \, \dd \beta \dd \mu^{*k} =\sum_{n\geq 0} \E_{\beta}(f(xb_{1}\dots b_{n+k})^2)-\E_{\beta}(f(xb_{1}\dots b_{n})^2)\leq k\sup f^2<\infty.$$  
It follows that $S$ is $\beta\otimes \mu^{*k}$-almost surely finite, which proves the lemma. 
\end{proof}

\begin{proof}[Proof of \Cref{CD-abelian-case}]
Let $\mu$ be a probability measure on $G$ whose support generates a dense subgroup.  Let $f :G\rightarrow \R$ be a left uniformly continuous bounded $\mu$-harmonic  function, and $(f_{b})_{b\in B}$ its limit functions defined above. Using that $G$ is abelian and \Cref{limit-decomp}, we see that for $\beta$-almost every $b\in B$, and $\mu$-almost every $g\in G$, one has $f_{b}=f_{b}(g \cdot)$, i.e. $f_{b}$ is $\supp \mu$-invariant. It follows that $f_{b}$ is constant, whence $f$ as well.
\end{proof}

In order to tackle the general case, we need a few preliminary results. Set $[u,v]=uvu^{-1}v^{-1}$ the commutator of  $u,v\in G$.

\begin{lemma}[Extra-invariance] \label{extra-inv-b}
Let $G$ be a  Lie group. Let $\mu$ be a probability measure on $G$, let $f:G\rightarrow \R$ be a   bounded $\mu$-harmonic function which is left uniformly
continuous. Write $(f_{b})_{b\in B}$ the associated limit functions. 

For $\beta$-almost every $b\in B$, for $\mu^{\otimes 2}$-almost every $(u,v)\in G^2$, for all accumulation points $z\in \Acc( (\Ad(b_{1}\dots b_{n})[u,v])_{n\geq 0} )$, we have $f_{b}=f_{b}(\cdot z)$
\end{lemma}

In the above, given a sequence $(x_{n})_{n\geq0}$ in a metric space $X$, we write $\Acc((x_{n})_{n\geq 0})$ its set of accumulation points, meaning the set of points $y\in X$ such that $(x_{n})_{n\geq0}$ has a subsequence converging to $y$. 

%\begin{lemma}[Extra-invariance] \label{extra-inv}
%Let $G$ be a  Lie group. Let $\mu$ be a probability measure on $G$ and $f:G\rightarrow \R$ a   bounded $\mu$-harmonic function which is left uniformly
%continuous. Let $G\times G\rightarrow G, (u,v)\mapsto z_{u,v}$ be a map such that for $\mu^{\otimes 2}$-almost every $(u,v)\in G^2$, for $\beta$-almost every $b\in B$, the point $z_{u,v}$ is an accumulation point of  $(\Ad(b_{1}\dots b_{n})[u, v])_{n\geq 0}$. Then  for $\mu^{\otimes 2}$-almost every $(u,v)$, for every $x\in G$, we have 
%$$f(x)=f(xz_{u,v}). $$
%\end{lemma}

%\noindent\emph{Remark}. 
%The parameter $z_{u,v}$ must not depend on the sequence of instructions $b$ selected by $\beta$. In other terms, $z_{u,v}$ is an accumulation point of $(\Ad(b_{1}\dots b_{n})[u, v])_{n\geq 0}$ uniformly for $\beta$-almost-every $b$. Below, as we seek to apply  \Cref{extra-inv}, it will be easy to check that the sequence $(\Ad(b_{1}\dots b_{n})[u, v])_{n\geq 0}$ has an accumulation point for any $b\in B$ (by compactness hypothesis), but we will need to work to see  that this accumulation point may be chosen independently of $b$ (restricted to a $\beta$-measure $1$ subset). 

\begin{proof} 
By \Cref{limit-decomp} we  have for $\beta$-almost every $b\in B$, for $\mu^{\otimes 2}$-almost every $u,v\in G$, every $x\in G$, 
\begin{equation}\label{trick-uv-vu}
f_{b}(x)=\lim f(x b_{1}\dots b_{n} u v) =\lim f(x b_{1}\dots b_{n} vu).
\end{equation}
The first equality in  \eqref{trick-uv-vu} can be rewritten 
\begin{equation}\label{uv-inside}
f_{b}(x)=\lim f(x \Ad(b_{1}\dots b_{n})[u, v]  b_{1}\dots b_{n}  vu).
\end{equation}
Let $z\in \Acc( (\Ad(b_{1}\dots b_{n})[u,v])_{n\geq 0} )$. Let $(n_{j})_{j\geq0}$ be a sequence of integers going to infinity such that $\Ad(b_{1}\dots b_{n_{j}})[u, v]  \to_{j} z$.  By \eqref{uv-inside} and left uniform continuity of $f$, we deduce
$$f_{b}(x)=\lim_{j} f(x z b_{1}\dots b_{n_{j}}  vu).$$
Using the second equality in \eqref{trick-uv-vu}, we get 
$$f_{b}(x)= f_{b}(x z).$$
\end{proof}

In view of \Cref{extra-inv-b}, we obtain that a bounded harmonic function $f$ is invariant by right multiplication under points $z$ which are accumulation points common to all sequences $(\Ad(b_{1}\dots b_{n})[u,v])_{n \geq 0}$ for a fixed $\mu^{\otimes 2}$-typical pair $(u,v)$, and $b$ in a $\beta$-full subset of $B$. 

\begin{corollary}\label{extra-inv}
Keep the notations of  \Cref{extra-inv-b}. For $\mu^{\otimes 2}$-almost every $(u,v)\in G^2$, for all $z\in G$ such that  for  $\beta$-almost every $b\in B$, we have 
$z\in \Acc( (\Ad(b_{1}\dots b_{n})[u,v])_{n\geq 0} )$, we have $f=f( \cdot z)$. 
\end{corollary}

\begin{proof} \Cref{extra-inv-b} implies $f_{b}=f_{b}(\cdot z)$ for $\beta$-almost every $b$, then $f=f(\cdot z)$ by integrating $b$ according to the law  $\beta$.

\end{proof}

From there it is easy to establish \Cref{CD-sufficient-Liegrp} in the nilpotent setting.

\begin{lemma}[Nilpotent case] \label{CD-nilp-setting}
Let $G$ be a connected nilpotent Lie group with $G^{[3]}$ bounded. Then $G$ is Choquet-Deny.
\end{lemma}

\begin{proof}
Consider $\mu$ a probability measure on $G$ whose support spans a dense subgroup, and $f$ a bounded $\mu$-harmonic   function which is  left uniformly continuous. By \Cref{CD-abelian-case}, it is sufficient to show that $f$ is invariant under $[G,G]$. 

Note that the relative compactness of $G^{[3]}$ implies that $G^{[3]}$ is central, hence $G$ has step at most $3$. The map $\phi:G^3\rightarrow G^{[3]}, (a,b,c)\mapsto [a,[b,c]]$ is thus a homomorphism in each variable $a,b,c$. In particular, given $u,v, g_{1},g_{2}\in G$, we  have 
$$\Ad(g_{1}g_{2})[u,v]= \phi(g_{1}g_{2},u,v)[u,v]=\phi(g_{1},u,v)\phi(g_{2},u,v)[u,v],$$
therefore acting by successive conjugations on $[u,v]$ amounts to moving $[u,v]$  by successive translations from the torus $\overline{G^{[3]}}$. By the Poincar\'e recurrence theorem, we deduce that 
for every $u,v\in G$, $w\in \{\id\}\cup\supp \mu$, and $\beta$-almost every $b\in B$, the element $\Ad(w)[u,v] $ is an accumulation point of the sequence $\Ad(b_{1}\dots b_{n})[u,v] $.

In view of \Cref{extra-inv}, we have justified that $f$ is right-invariant under the closed subgroup $I$ generated by $\{\Ad(w)[u,v]  \,:\, u,v\in\supp \mu, w\in \{\id\}\cup \supp \mu\}$. Note that $I$ contains  $(\Ad(w)[u,v] ) [u,v]^{-1}=[w, [u,v]]$. Recalling that $\phi$ is a homomorphism in each variable,  and the assumption on $\supp \mu$, we deduce that $I$ contains $G^{[3]}$. Hence $I \cap G^{[2]}$ is normal, so $I$ as well by passing to the closure. Using again the assumption on $\supp \mu$, we obtain that $G/I$ is abelian, so $I\supseteq G^{[2]}$ and this finishes the proof.
\end{proof}

We now aim to deal with  general groups. For the nilpotent case above, we relied on  Poincar\'e's recurrence theorem to obtain some accumulation points which are common to all sequences $(\Ad(b_{1}\dots b_{n})[u, v])_{n\geq 0}$ parametrized by a fixed $[u,v]$, and $b$ in a  $\beta$-full set. This allowed us to apply   \Cref{extra-inv} and get some extra invariance for bounded harmonic functions. For the general case, we show accumulation sets do not depend on instructions via the following result.

\begin{lemma}\label{accumulation}
Let $(H, \mu)$ be a Lie group endowed with a Liouville probability measure. Let $Y$ be a  metric space endowed with a continuous\footnote{The continuity requirement means that the action map $H\times Y \rightarrow Y, (h,y)\mapsto hy$ is continuous.} action of $H$. 
Let $y\in Y$ with $Hy$ relatively compact. There exists a non-empty closed  $H$-invariant set $Z_{y}\subseteq \overline{Hy}$ such that
 for $\mu^{\otimes \N^*}$-almost every $b\in H^{\N^*}$,  the set of accumulation points of the sequence $(b_{1}\dots b_{n}y)_{n\geq0}$ is exactly $Z_{y}$.
 \end{lemma}

\noindent\emph{Remark}. In general, $Z_{y}$ might not contain $y$. For example, take $H=\R$ acting on the closed Poincar\'e disk via $t.x=g_{t}.x$ where $g_{t}$ is the geodesic flow. Let $\mu= \leb_{[1,2]}$. Then every point has bounded $H$-orbit (as the closed disk is compact), and $b_{1}\dots b_{n}y\to y_{\infty}=:\lim_{+\infty}g_{t}y$ almost-surely. 

 \begin{proof}
 Let $U$ be an open ball in $Y$. Given $h\in H$, set $f_{U}(h)$ to be the $\mu^{\otimes \N^*}$-probability that for infinitely many $n \geq1$, we have $hb_{1}\dots b_{n}y\in U$. Observe that $f_{U}$ is a bounded $\mu$-harmonic function on $H$. As $(H, \mu)$ is Liouville, $f_{U}$  has to be Haar-a.e. constant, say $f_{U}=c_{U}$ almost everywhere. 
 In fact $c_{U}\in \{0, 1\}$.  To see why, assume $c_{U}>0$. We have established that for Haar-almost every $h\in H$, there exists $N_{h}\in \N$ such that for   $b=(b_{i})_{i\geq 1}$ varying according to $\mu^{\otimes \N^*}$, we have with probability at least $c_{U}/2$ that  the sequence $(h b_{1}\dots b_{n}y)_{n\geq 1}$ meets $U$ for some $1\leq n\leq N_{h}$. This allows us to argue conditionally to previous steps of the (right)  $\mu$-random walk on $H$, and finally get $c_{U}=1$. 
 
 Now set 
 $$Z_{y}:= {\{ z\in Y\,:\, c_{U}=1  \text{ for every open ball $U$ containing $z$}  \}}.$$
 Equivalently, $Z_{y}$ is the complement of the union of the open sets $U$ such that $c_{U}=0$, in particular, $Z_{y}$ is closed and included in $\overline{Hy}$. Moreover, $Z_{y}$ is $H$-invariant due to the relation $f_{h_{1}U}(h_{2})= f_{U}(h_{1}^{-1}h_{2})$   ($h_{1},h_{2}\in H$) which implies $c_{U}=c_{h_{1}U}$. 
 
 It remains to show that for $\mu^{\otimes \N^*}$-almost every $b\in H^{\N^*}$, the accumulation set of $(b_{1}\dots b_{n}y)_{n\geq1}$ is $Z_{y}$. To see this, consider  $(U_{k})_{k\geq 1}$ a countable basis of $\overline{Hy}$. Fix $h_{0}\in H$ such that $f_{U_{k}}(h_{0})=c_{U_{k}}$ for all $k\geq 1$. By the preceding paragraph,  for $\mu^{\otimes \N^*}$-almost every $b$,  for all $k\geq 1$, the sequence $(h_{0}b_{1}\dots b_{n}y)_{n\geq1}$ meets $U_{k}$ infinitely often if $c_{U_{k}}=1$, finitely often otherwise. For such $b$, this yields $Z_{y}$ is the accumulation set of $(h_{0}b_{1}\dots b_{n}y)_{n\geq1}$. By $H$-invariance of $Z_{y}$, it is also the accumulation set of $(b_{1}\dots b_{n}y)_{n\geq1}$. This finishes the proof.

 \end{proof}

In order to apply \Cref{accumulation}, we need the measure $\mu$ to be Liouville on $H$.  This assumption seems counterproductive because the goal of the section is to establish that probability measures are Liouville on certain groups. In fact, we will still be able to apply it, not directly to the group $G$ under study, but to a quotient group, which is allowed up to arguing by induction on the dimension. For that, we need the next lemma, which asserts that (certain) points in $G$ have a centralizer of positive dimension.

\begin{lemma}[Big centralizer] \label{centralizer-bnd}
Let $G$ be a connected Lie group. Let $N$ be its nilpotent radical, $T_{N}$ the largest torus in $N$, and $s\geq 1$ the step of $N$. Then $T_{N}$ is central in $G$. If $T_{N}=\{\id\}$ then every  $x\in G$ such that $\Ad(N^{[s]})x$ is bounded must commute with $N^{[s]}$.
\end{lemma}

\begin{proof}
Note $T_{N}$ is normal in $G$. On the other hand, $G$ is connected and the group of automorphisms of a torus is discrete. Therefore $\Ad(G)\acts T_{N}$ is trivial, i.e. $T_{N}$ is central.

Let us check the second claim.  As $T_{N}=\{\id\}$, the group $N$ is simply connected (\Cref{max-comp-in-nilp}). Therefore the abelian subgroup $N^{[s]}$ can be seen as a vector space. Note $N^{[s]}$ is also invariant under automorphisms of $G$, and in particular under $\Ad(x)$. The action $\Ad(x)\acts N^{[s]}$ is linear (because it is continuous and commutes with integer multiplication). By assumption the subset $\{n x n^{-1} x^{-1} \}_{n\in N^{[s]}} \subseteq N^{[s]}$ is bounded in $G$. As the inclusion map $N^{[s]}\hookrightarrow G$ is proper, it is also bounded in $N^{[s]}$, and recalling the group $N^{[s]}$ is a vector space, we get
$$\sup_{n \in N^{[s]}} \|(\id_{N^{[s]}}-\Ad(x))(n)\|<\infty.$$
 A vector space endomorphism which is uniformly bounded must be zero, so $x$ commutes with $N^{[s]}$. 
\end{proof}

%\noindent\emph{Remark}. One may ask whether the assumption $T=\{\id\}$ in the second claim is necessary. Note that an automorphism of $T\times V$ is of the form $\psi(v,z)=(l(v), c(v)+a(z))$ where $l\in \Aut(V)$, $a\in \Aut(T)$ and $:V\rightarrow T$ is a morphism (and conversely). It satisfies that  $\id-\psi$ is uniformly bounded if and only if $l=\id$. So we would need to check that $c=0$ in the context $\psi=\Ad(x)$ to obtain that $x$ commutes with $N^{[s]}$ regardless of whether $T$ is trivial or not.

We now work to guarantee that the extra-invariance provided by  \Cref{extra-inv}, whenever applicable,  is not trivial. We start with a preparatory lemma. 
 
\begin{lemma} \label{G/R-compact}
Let $G$ be as in \Cref{CD-sufficient-Liegrp}, let $R$ be the radical of $G$.  The semisimple group $G/R$  is compact. 
\end{lemma}

\begin{proof}
Set $H:=G/R$ and $L=H/Z(H)$. As $L$ is semisimple and connected, we have that $[L,L]=L$, therefore every element in $L$ has bounded conjugacy class. Assume by contradiction that $H$ is not compact. The projection map $H\rightarrow L$ is a covering, therefore by  Weyl's theorem (\Cref{Weyl}), the group $L$ is not compact either. 
Let $L=KAU$ be an Iwasawa decomposition for $L$. As $L$ is semisimple and center free, non-compactness implies $A,U\neq \{\id\}$. Choosing $a$ in the Weyl chamber dilating $U$, we have for any $u\in U\smallsetminus \{\id\}$ that $\Ad(a^n)u\to \infty$  in $U$ as $n$ tends to $+\infty$. The inclusion $U\hookrightarrow L$ being proper, this shows $u$ has unbounded conjugacy class. Contradiction.
\end{proof}

\begin{lemma} \label{convid} 
Let $G$ be as in \Cref{CD-sufficient-Liegrp}. Let $x\in [G,G]$. Let $\mu$ be a probability measure on $G$.  Assume that for $\mu^{\otimes \N^*}$-almost every $(b_{i})_{i\geq1}$, one has 
$$\Ad(b_{1}\dots b_{n})x\rightarrow \id. $$
Then $x=\id$.
\end{lemma}

%\noindent\emph{Remark}. 
%The proof works under the weaker assumption that  the simple factors of $G/R$ are compact, $x\in G$ has bounded $\Ad(G)$-orbit, and  

\begin{proof}
Let $R$ be the radical of $G$. We first check that $x\in R$.  Set $H:=G/R$. 
Note that $\Ad(G)\acts H$ by automorphisms. Note $H$ is a compact connected Lie group  (\Cref{G/R-compact}) and $G$ is connected, therefore every element of $G$ acts by conjugation by an element of $H$ (\Cref{autom-K}).  Therefore, for $\beta$-almost every $b\in B$, the sequence of automorphisms $\Ad(b_{1}\dots b_{n}) \acts H$ has a subsequence converging to the identity. From the assumption on $x$, we deduce $x\in  R$.

Let $R'\subseteq R$ such that $R/R'$ is the biggest vector space  quotient of $R$, see \Cref{vs-radical}. We  show $x\in R'$. Let $V_{x}\subseteq R/R'$ be the subspace generated by $\Ad(G)x \mod R'$. By hypothesis, $\Ad(G)$ acts on  $V_{x}$ with bounded orbits, whence by isometries. One can then argue as in the previous paragraph to get that $x \mod R'$ is a limit point of $\Ad(b_{1} \dots b_{n})(x \mod R')$ in $V_{x}$, whence $x\in R'$ due to the assumption on $x$. 

We deduce that $x\in \overline{[R,R]}$. Indeed, $R'/ \overline{[R,R]}$ is a torus so $\Ad(G)\acts R'/ \overline{[R,R]}$ is trivial by connectedness of $G$ and \Cref{autom-K}. The assumption on $x$ then  implies  $x\in \overline{[R,R]}$. 

We can repeat this argument with $\overline{[R,R]}$ in the place of $R$ and reach $x=\id$ in a finite number of steps. Note here we do not consider the successive derived subgroups of $R$, but  slightly bigger groups, obtained by iterating the map $\Phi : S\mapsto \overline{[S,S]}$. The sequence  reaches $\{\id\}$ in finite time because $R$ is solvable and $\Phi^k(R)= \overline{{^{[k+1]}R}}$ for all $k$ (where $({^{[k]}R})_{k}$ denotes the derived series of $R$).

\end{proof}

We may now conclude the proof of the Choquet-Deny property for connected Lie groups such that all commutators have bounded conjugacy class.

\begin{proof}[Proof of \Cref{CD-sufficient-Liegrp}] 
We argue by induction on the dimension of $G$. The result is clear when $\dim G=1$ because $G$ is then abelian, a case we already dealt with in \Cref{CD-abelian-case}. We consider $\mu$ a probability measure on $G$ whose support generates a dense subgroup. We let $f$ be a left uniformly continuous bounded $\mu$-harmonic function. \emph{We assume by contradiction $f$ non constant.}

 Let $u,v\in G$. Let $C_{u,v}\subseteq G$ be the largest normal closed subgroup which centralizes $[u,v]$.  We claim that $G/C_{u,v}$ is Choquet-Deny. To see that, let $N$ be the nilpotent radical of $G$. If $N=\{\id\}$, then the  radical $R$ of $G$ is trivial as well, so $G$ is compact by \Cref{G/R-compact}. As $f$ is non-constant, we have a contradiction with the Choquet-Deny property for compact Lie groups (\Cref{CD-compact-case}). Hence $N\neq\{\id\}$. One may then apply  \Cref{centralizer-bnd} and the boundedness of $\Ad(G)[u,v]$, to see that $\dim C_{u,v}>0$, which by induction  hypothesis, implies that $G/C_{u,v}$ is Choquet-Deny. We have thus established the claim. 
 
We now exhibit a normal closed subgroup $I\subseteq G$ such that $f$ is $I$-invariant. The fact that $G/C_{u,v}$ is Choquet-Deny allows us to apply \Cref{accumulation} with $H=G/C_{u,v}$, and the basepoint $[u,v]$. We write $Z_{[u,v]}\subseteq \overline{\Ad(G)[u,v]}$ the associated accumulation set. We deduce from \Cref{extra-inv} that there exists a set $E\subseteq G^2$ of full $\mu^{\otimes 2}$-measure such that $f$ is right-invariant under  the closed subgroup generated by $\cup \{Z_{[u,v]}\,:\,(u,v)\in E\}$. We denote this subgroup by $I$. Note that $I$ is  normal in $G$ because \Cref{accumulation} guarantees that each $Z_{[u,v]}$ is $\Ad(G)$-invariant. 

We check that $I$ is discrete central and  $I\neq \{\id\}$.  If $\dim I>0$, then seeing $f$ as a function on $ G/I$, the induction hypothesis yields that $f$ is constant, contradiction. So necessarily $\dim I=0$, i.e.   $I$ is a discrete, normal (whence central) subgroup. 
 Assume by contradiction $I=\{\id\}$. By \Cref{convid}, we have $[u,v]=\id$ for all $(u,v)\in E$, so $G$ must be abelian, in which case $f$ is constant because abelian groups are Choquet-Deny (\Cref{CD-abelian-case}). This contradicts our assumption, hence $I\neq \{\id\}$.

Seeing $f$ as a harmonic function on $I\backslash G$, and  repeating the same procedure,  we deduce that $f$ is invariant under a strictly increasing sequence of normal discrete subgroups $(I_{n})$. Note they are all central, because $G$ is connected so it can only act trivially on discrete sets.  Set $I_{\infty}:=\overline{\cup I_{n}}$. By \Cref{discrete-central-fg} below, we must have $ \dim I_{\infty}>0$.   We can see $f$  as a harmonic function on $G/I_{\infty}$. By induction hypothesis, $f$ must be constant. Contradiction. 
\end{proof}

The next lemma was used in the previous proof. 

\begin{lemma}\label{discrete-central-fg} 
Let $G$ be a connected Lie group and let $D$ be a discrete central subgroup of $G$.
Then $D$ is a finitely generated abelian group. Consequently, every increasing
sequence of subgroups of $D$ stabilizes.
\end{lemma}

\begin{proof}
Since $D$ is discrete and central, the quotient map
\[
G \longrightarrow G/D
\]
is a covering map of connected Lie groups. The associated homotopy exact
sequence yields a surjection
\[
\pi_1(G/D)\longrightarrow D.
\]
The fundamental group of a connected Lie group is finitely generated abelian; see, for instance,
\cite[Theorem~14.3.11]{HilgertNeeb2012} and
\cite[Chapter~23]{Bump2013}.
Hence $D$, being a quotient of $\pi_1(G/D)$, is finitely generated abelian.

The last assertion follows because every finitely generated abelian group is
Noetherian.
\end{proof}

\section{An escape property }\label{SEP-section}

We show that  Lie groups for which there exists a commutator with unbounded conjugacy class satisfy in fact  a much stronger escape property. We introduce this property below as $\SEP$. In the next section, it  will play a key role to show  that such groups are not Choquet-Deny. 

\bigskip

Given \(m\geq1\), let \(\Omega_m\) be the collection of tuples
\[
 \omega=((w_1,\varepsilon_1),\ldots,(w_r,\varepsilon_r)),
\]
where $r\in \N$, where  $w_1,\dots, w_r$ are pairwise distinct elements of $\{1,\ldots,m\}$, and $\varepsilon_1, \dots, \eps_{r} \in\{-1,1\}$. The case $r=0$ corresponds to the empty tuple $\omega=\emptyset$. 

For \(\ux=(x_1,\ldots,x_m)\in G^m\), and $\omega,\omega'\in \Omega_{m}$, set
\[
 \ux_\omega=x_{w_1}^{\varepsilon_1}\cdots x_{w_r}^{\varepsilon_r},
 \qquad
 \ux_{\omega,\omega'}=\ux_\omega\ux_{\omega'}^{-1}.
\]
For the empty tuple, the convention is $\ux_{\emptyset}=\id$.

\begin{definition}\label{def-SEP} 
We say that a tuple \(\ux\in G^m\) satisfies the \emph{simultaneous escape
property} \(\SEP\) if there exists a sequence \((\sigma_k)_{k\geq0} \in G^\N\) such
that for every $\omega\neq \omega'\in \Omega_{m}$, and every $(\eps, \eps')\in\{-1,1\}^2\smallsetminus \{(-1,1)\}$, we have 
\begin{align*} 
 \sigma_k^{\eps}  \ux_{\omega,\omega'} \sigma_k^{\eps'}&\longrightarrow \infty \quad\quad\quad \quad \sigma_k^{-1}  \ux_{\omega,\omega'} \sigma_k \longrightarrow \id \text{ or }  \infty.
\end{align*}
 We say
that \(G\) satisfies \(\SEP\) if, for every \(m\geq 1\), the 
property \(\SEP\) holds for  a  dense subset of tuples $\ux\in G^m$.
\end{definition}

 In the above, given a sequence $(y_{k})_{k\geq 0}$ in $G$, the notation 
$ y_k\to \infty$  means that  \((y_k)_{k\geq0}\) eventually leaves every compact subset of $G$. We will refer to $(\sigma_{k})_{k\geq0}$ as a ``switching sequence'' for $\ux$.

\bigskip
\noindent\emph{Remark}.
If $G$ satisfies $\SEP$ then we may further require switching sequences $\sigma_{k}$ to satisfy $\sigma_{k}^2\to\infty$. Indeed, for every $m\geq 1$, consider $(x_{1}, \dots, x_{m}, t) \in G^{m+1}$ satisfying $\SEP$, with some switching sequence $(\sigma_{k})_{k}$. Then it is straightforward to check that the new sequence $\tau_{k}=\sigma_{k}t$ switches $(x_{1}, \dots, x_{m})$ and furthermore satisfies $\tau_{k}^2\to \infty$.

\bigskip

Our goal for the section is to show that unbounded conjugacy class for a commutator can be upgraded to  $\SEP$.

\begin{proposition}[Simultaneous escape] \label{prop-sim-switch}
Let $G$ be a  connected Lie group. The next conditions are equivalent:
\begin{itemize}
\item[a)] there exists $x\in [G,G]$ with $\Ad(G)x$ unbounded
\item[b)]  there exists $x\in \overline{[G,G]}$ with $\Ad(G)x$ unbounded
\item[c)]  a quotient of $G$ satisfies $\SEP$.
\end{itemize}
\end{proposition} 

The implications $a)\implies b)$ and $c)\implies a)$ are trivial. We will only need to check $b)\implies c)$. 
In the case where $G$ is nilpotent or semisimple, an inspection of the proof shows  that $b)$ implies a strong form of $\SEP$: for all $m\geq1$, for an open dense set of $\ux\in G^m$, there exists $(\sigma_{k})\in G^\N$ such that  $ \sigma_k^{\pm 1}  \ux_{\omega,\omega'} \sigma_k^{\pm 1}\to \infty$, i.e. we do not need to allow the possibility $\sigma_k^{-1}  \ux_{\omega,\omega'} \sigma_k\to \id$. However, in the solvable case, the alternative $\sigma_k^{-1}  \ux_{\omega,\omega'} \sigma_k \to \id \text{ or } \infty$ cannot be avoided. This is why we formulate $\SEP$ this way. This constraint for solvable groups can be seen in the following basic example.

\begin{example}\label{ex-affinegroup}
Let $G$ be the group of orientation-preserving affine transformations of $\R$, given by
\[
 G=\R\ltimes\R,\qquad (a,b)(c,d)=(a+c,b+e^a d),
\]
For $\sigma=(r,s)\in G$, and $x=(0, t)\in [G,G]$, we have 
$$
 \sigma x\sigma^{-1}=(0, e^r t),
 \qquad
 \sigma^{-1} x\sigma=(0, e^{-r}t ).
$$
Therefore, any sequence $(\sigma_{k})$ such that $ \sigma_{k} x\sigma^{-1}_{k}\to \infty$ must satisfy that  $\sigma^{-1}_{k} x\sigma_{k} \to \id$. 
\end{example}

%Of course, the converse holds: $\SEP$ implies that some commutator has unbounded conjugacy class. Therefore, \Cref{prop-sim-switch} states an equivalence.

\bigskip

The remainder of the section is devoted to the proof of \Cref{prop-sim-switch}. The next result is a straightforward preliminary, which deserves to be recorded explicitly for it will be used on multiple occasions below.
\begin{lemma} \label{avoidance}
Let $M_{1}, M_{2}$ be real-analytic manifolds with $M_{1}$ connected. Let $M_{2}'\subseteq M_{2}$ be a  real-analytic submanifold which is closed for the ambient topology. Let $f:M_{1}\rightarrow M_{2}$ be a real-analytic map. Then $f^{-1}(M_{2}')$ either has empty interior or coincides with $M_{1}$.
\end{lemma}

%In the latter case, we must have $H\supseteq [G,G]$ if $w,w'$ involve exactly the same letters, and $H=G$ otherwise.

Note we need $M_{2}'$ to be closed (topologically) in order to rule out situations such as $M_{1}=M_{2}=\R$, $M_{2}'=(0,+\infty)$, $f(x)=x$.

\begin{proof}[Proof of \Cref{avoidance}]
Let $U$ be the largest open set of $M_{1}$ which is included  in $f^{-1}(M_{2}')$.  Assuming $U\neq \emptyset$, we need to show $U=M_{1}$. By connectedness of $M_{1}$, it is sufficient to show that $U$ is closed. Let $x\in \overline{U}$. As $M_{2}'$ is closed, we have $f(x)\in M_{2}'$. As $M_{2}'$ is a real-analytic submanifold, we may consider a real analytic chart $\psi : V\rightarrow \R^d$ defined on a neighborhood $V$ of $f(x)$, and such that $\psi (V\cap M_2')=\psi(V)\cap \R^{d'}\times \{0\}^{d-d'}$. Now, $\psi \circ f$ is well defined and real-analytic on a small connected open neighborhood of $x$, with values in $\R^{d'}\times \{0\}^{d-d'}$ when restricting to some non-empty open set. By  the identity theorem for analytic maps, $\psi \circ f$ sends this whole neighborhood of $x$ to $\R^{d'}\times \{0\}^{d-d'}$, therefore $x\in U$. We have thus established that $U$ is closed, thus finishing the proof. 
\end{proof}

\Cref{avoidance} will be helpful to show that products of the form $\ux_{\omega,\omega'}:=\ux_{\omega}\ux_{\omega'}^{-1}$ generally avoid  Lie subgroups, unless there is a trivial obstruction. Indeed, let $H\subseteq G$ be a  Lie subgroup, let $m\geq 1$, $\omega\neq \omega'\in \Omega_{m}$, consider the open set
$$O_{\omega,\omega'} :=\{\ux \in G^m\,:\, \ux_{\omega,\omega'} \notin H\}.$$
Then in view of  \Cref{analyt-struct} and \Cref{avoidance}, we have either $O_{\omega,\omega'}$  dense or empty.
To rule out the case where $O_{\omega,\omega'}$ is empty, we will rely on the next lemma stating that the set $\{\ux_{\omega,\omega'}\,:\, \ux \in G^m\}$ is rather large.

Given $\omega=(w_{j}, \eps_{j})_{j=1}^k\in \Omega_{m}$, we write $\omega_{ab}$ for the \emph{unordered} collection of pairs $(w_{j}, \eps_{j})$ appearing in $\omega$. For example taking $\omega=((2,-1),(3,1))$, $\omega'=((3,1),(2,-1))$ we have $\omega\neq \omega'$ but $\omega_{ab}=\omega'_{ab}$.

\begin{lemma} \label{Iww'-large}
Let $\omega,\omega'\in \Omega_{m}$ with $\omega\neq \omega'$. Write $I_{\omega,\omega'}:=\{\ux_{\omega,\omega'}\, :\, \ux \in G^m\}$. 
If $\omega_{ab}\neq \omega'_{ab}$ then $I_{\omega,\omega'}$ contains a neighborhood of $\id$ in $G$. If $\omega_{ab}= \omega'_{ab}$ then $I_{\omega,\omega'}$ contains every commutator $[g,h]$ where $g,h\in G$. 
\end{lemma}

\begin{proof}  Assume $\omega_{ab}\neq \omega'_{ab}$. If $\omega$ and $\omega'$ do not have the same set of letters $w_{k}\in \{1, \dots, m\}$, then $I_{\omega,\omega'}=G$. Otherwise, the same letters appear but they do not have the same signed exponent. In this case $I_{\omega,\omega'}$ contains all $g^2$ where $g\in G$ which implies the result because the map $g\mapsto g^2$ is a submersion at $\id$.

Assume $\omega_{ab}= \omega'_{ab}$.  We may assume all signed exponents equal to $1$. Recalling $\omega\neq \omega'$, we get that $\omega,\omega'$ contain two letters $w_{k_{1}}$ and $w_{k_{2}}$ which appear in different orders, so $I_{\omega,\omega'}$ contains every commutator. 
\end{proof}

We now engage in the proof of $b)\implies c)$ in \Cref{prop-sim-switch}.  We will distinguish several cases depending on the structure of $G$. 

\subsection{The case where $G$ is nilpotent} \label{nilpSEP}

We validate \Cref{prop-sim-switch} in the nilpotent setting.

\begin{proposition}[Simultaneous escape - nilpotent case] \label{nilp-SEP}
Let  $G$ be a nilpotent connected Lie group.  The  conditions $a)$, $b)$, $c)$ in  \Cref{prop-sim-switch} are mutually  equivalent and hold if and only if  $[G,[G,G]]$ is unbounded.
\end{proposition}

The next lemma is the core of the proof.

\begin{lemma} \label{nilp-SEP-core}\label{nilpotent-SSEP}
Let $G$ be a nilpotent connected  Lie group. If $[G,[G,G]]$ is unbounded, then a quotient of  $G$  satisfies $\SEP$.
\end{lemma}

\begin{proof}
Let $T_{G}$ be the largest compact (central)  subgroup of $G$ (see \Cref{max-comp-in-nilp}). 
We show $G/T_{G}$ satisfies $\SEP$. We may assume $T_{G}=\{\id\}$, i.e. $G$ is simply connected. 
%In particular, $[G,G]$ is now closed.

Let $m\geq 1$. For $\omega\neq \omega'\in \Omega_{m}$, write 
$$O_{\omega,\omega'} :=\{\ux \in G^m\,:\, \ux_{\omega,\omega'}  \notin Z(G)\}.$$
Note $[G,G]$ is not fully contained in $Z(G)$ because $G^{[3]}$ is non-trivial by assumption. Therefore, by \Cref{Iww'-large}, we have $O_{\omega,\omega'}\neq \emptyset$, and by \Cref{avoidance}, this means that $O_{\omega,\omega'} $ is  dense (and open). The set  $O=\cap_{\omega \neq \omega'} O_{\omega,\omega'}$ is thus also open and dense. 

Let $\ux\in O$. Write $\{y_{i} \,:\, i\in I\} =\{\ux_{\omega, \omega'} \,:\, \omega\neq \omega'\}$. As no $y_{i}$ belongs to $Z(G)$, we may find $g$ such that $g^{\eps} y_{i} g^{\eps'}\neq y_{i} $ for every $i\in I$, every $\eps, \eps'\in \{-1,1\}$. We then  observe that for each $i$ and $\eps, \eps'$, we have $g^{\eps k} y_{i} g^{\eps 'k}\to_{k} \infty$.  Indeed, identifying $G$ with its Lie algebra $\kg$ via the exponential map (\Cref{fact-nilp-spc}), we see that each map $\N\rightarrow \kg, k\mapsto g^{\eps k} y_{i} g^{\eps 'k}$ is polynomial, and non-constant by choice of $g$. Therefore this map is proper, which shows $\ux$ satisfies $\SEP$ with switching sequence $\sigma_{k}=g^k$.
\end{proof}

We may now easily derive \Cref{nilp-SEP}.

\begin{proof}[Proof of \Cref{nilp-SEP}]
If $[G,[G,G]]$ is unbounded, then  $c)$ holds in view of \Cref{nilp-SEP-core}.  If $b)$ holds, then   $[G,\overline{[G,G]}]$ is clearly unbounded, which transfers to $\overline{[G,[G,G]]}$, then $[G,[G,G]]$. This justifies $b)\implies c)$ in \Cref{nilp-SEP}, therefore  $a)$, $b)$, $c)$ are mutually equivalent, and furthermore equivalent to $[G,[G,G]]$ being unbounded.
\end{proof}

\subsection{The case where $G$ is solvable}

We validate \Cref{prop-sim-switch} for solvable Lie groups. 

\begin{proposition}[Simultaneous escape - solvable case]  \label{solv-SEP}
Let  $G$ be a solvable connected Lie group.  The conditions $a)$, $b)$, $c)$ in \Cref{prop-sim-switch} are mutually equivalent and hold if and only if,  setting $S=\overline{[G,G]}$, we either have that the action of $\Ad(G)$ on the maximal quotient vector space $S/S'$ (equivalently $S/\overline{[S,S]}$) has some unbounded orbit, or $S'$ (equivalently $[S,S]$) is unbounded. 
%The action of $\Ad(G)$ on the maximal quotient vector space $S/S'$ where $S=\overline{[G,G]}$  is by isometry (for some Euclidean norm), and $S'$ (equivalently $\overline{[S,S]}$) is compact. 
\end{proposition}

We recall that the {largest vector space quotient $S/S'$ of $S$} was formally  introduced in \Cref{vs-radical}.

\bigskip

We start with  a useful criterion for a solvable Lie group to be nilpotent. 

\begin{lemma}\label{solv->nil}
Let $G$ be a solvable connected  Lie group. Set $S=\overline{[G,G]}$, and $S/S'$ the largest  vector space quotient of $S$.  If the representation $\Ad(G) \acts S/S'$ is trivial, then $G$ is nilpotent.
\end{lemma}

\begin{proof}
It is known that $S$ is nilpotent (consequence of Lie's theorem). Let $T_{S}$ be the largest torus in $S$. As $T_{S}$ is invariant by automorphisms of $G$, it is in particular a normal subgroup. As $G$ is connected, the action of $G$ by conjugation on $T_{S}$ is trivial (\Cref{autom-K}), i.e. $[G,T_{S}]=\{\id\}$. Note also $T_{S}\subseteq S'$. Therefore, in order to prove the lemma, we may replace $G$ by $G/T_{S}$, i.e. we may assume $T_{S}$ trivial, which means $S$ is simply connected. In this case $S'=[S,S]$. The assumption   that $\Ad(G) \acts S/S'$ is trivial now means $[\kg, \ks]_{\kg}  \subseteq [\ks, \ks]_{\kg}$ where $\ks$ stands for  the Lie algebra of $S$. Using the Jacobi identity, we deduce by induction $[\kg, \ks^{[k]}]_{\kg}\subseteq  \ks^{[k+1]}$ for all $k$. As $\ks$ is nilpotent, the lemma follows. 
\end{proof}

We now deal with solvable Lie groups whose second derived group is unbounded. This situation is illustrated by $G_{1}$ in \Cref{ex-SO2Heis}.

\begin{lemma}\label{solvable-SEP1}
Let \(G\) be a connected solvable Lie group, set $S=\overline{[G,G]}$. If $[S,S]$ is unbounded then a quotient of \(G\) satisfies
\(\SEP\).
\end{lemma}

\begin{proof}
By \Cref{nilp-SEP}, we may assume $G$ non-nilpotent. In particular, by \Cref{solv->nil},  the action $\Ad(G)\acts S/S'$ is non-trivial. Recall that $S$ is nilpotent (by Lie's theorem). Up to quotienting $S$ by its largest torus, we may suppose $S$ simply connected (\Cref{max-comp-in-nilp}).  Let $m\geq1$. For $\omega\neq \omega'\in \Omega_{m}$, set
$$
O_{\omega,\omega'} := \left\{
    \begin{array}{ll}
     \{\ux\in G^m \,:\, \ux_{\omega, \omega'}\in S\smallsetminus Z(S) \} & \mbox{if } \omega_{ab} =\omega'_{ab} \\
       \{\ux\in G^m \,:\, \Ad(\ux_{\omega, \omega'})_{|S/S'} \neq  \pm \id_{S/S'}\}  & \mbox{if  $\omega_{ab} \neq \omega'_{ab}$.}
\end{array}
\right.
$$
Let us check $O_{\omega,\omega'}$ is nonempty. In the case where $\omega_{ab}=\omega'_{ab}$, having  $O_{\omega,\omega'}=\emptyset$ would imply that every commutator $[a,b]$ ($a,b\in G$) is in $Z(S)$ (\Cref{Iww'-large}), which in turn would yield $S$ is abelian, contradicting the standing assumption.  In the  case where $\omega_{ab}\neq \omega'_{ab}$,  having $O_{\omega,\omega'}=\emptyset$ would imply that $\Ad(G)_{|S/S'}=\{\id_{S/S'}\}$, i.e. $G$ is nilpotent (see \Cref{solv->nil}), which we have excluded from the start. We have thus checked $O_{\omega,\omega'}$ is nonempty. By \Cref{avoidance}, each $O_{\omega,\omega'}$ is dense (and open), therefore the set $O:=\cap_{\omega\neq \omega'}O_{\omega,\omega'}$ is also open and dense in $G^m$. 

Let $\ux \in O$. We show $\ux$ satisfies $\SEP$. By definition of $O$, we may find $s\in S$, such that for  $\omega\neq \omega'\in \Omega_{m}$, $\eps,\eps'\in \{-1,1\}$, writing $v=s\mod S'$, we have 
$$
    \begin{array}{ll}
     s^{\eps}\ux_{\omega, \omega'}s^{\eps'} \neq \ux_{\omega, \omega'} & \mbox{if } \omega_{ab} =\omega'_{ab} \\
        \Ad(\ux_{\omega, \omega'})_{|S/S'} (v)\neq \pm v & \mbox{if  $\omega_{ab} \neq \omega'_{ab}$.}
\end{array}
$$ 
Indeed, for fixed $\omega,\omega'$, the associated requirement holds for a dense open set of $s\in S$, whence we may consider an element $s$ satisfying them all. We now check that $(s^k)_{k}$ is a switching sequence for $\ux$. Let $\omega,\omega'\in \Omega_{m}$. If $\omega_{ab} =\omega'_{ab}$, note the map $k\mapsto s^{\eps k}\ux_{\omega, \omega'}s^{\eps'k }$ is polynomial (\Cref{fact-nilp-spc}) and non-constant by choice of $s$, therefore it is proper: $s^{\eps k}\ux_{\omega, \omega'}s^{\eps'k }\to_{k} \infty$ as $k$ tends to $+\infty$. If $\omega_{ab} \neq \omega'_{ab}$, note that in the vector space $S/S'$, 
$$\| \Ad(\ux_{\omega, \omega'})_{|S/S'} (kv) \pm kv\| \to_{k} +\infty$$
and lifting to $G$ this justifies
$$ s^{\eps k} \ux_{\omega, \omega'} s^{\eps' k} \to_{k} \infty,$$
thus finishing the proof.
\end{proof}

To deal with other solvable groups, we  record the following linear-algebraic fact. 

\begin{lemma}\label{abelian-linear-SSEP} Let \(A\) be a connected abelian Lie group, acting linearly on a finite-dimensional real vector space $V$. Assume the action has an unbounded
orbit. Then there exists an \(A\)-invariant subspace $W\subseteq V$, a proper subspace $L\subsetneq V/W$, a sequence \((a_k)\) in \(A\) such that, for the quotient action on $V/W$, every $u\in V/W\smallsetminus L$ satisfies
\[
 a_k u\longrightarrow\infty,
 \qquad
 a_k^{-1} u\longrightarrow 0 \text{ or } \infty.
\]
\end{lemma}

\begin{proof}
We can assume $V=\R^d$ ($d\geq1$). Set $V_{\C}=\C^d$ its complexification. Then $V_{\C}$ is a complex representation of $A$. As $A$ is connected abelian, its action on $V_{\C}$ can be put in triangular form. More precisely, there exists a set of homomorphisms $\Lambda=\{\lambda : A\rightarrow \C^*\}$ such that 
$$V_{\C}=\oplus_{\lambda\in \Lambda} V_{\C,\lambda}$$
where each $V_{\C,\lambda}$ is a non-trivial $A$-invariant subspace, and for each $a\in A$, we have $a_{|V_{\C,\lambda}}=\lambda(a)\id_{V_{\C,\lambda}}+ N_{\lambda}(a)$ with $N_{\lambda}(a)\in \End(V_{\C,\lambda})$ nilpotent.
Note the family of weights $\Lambda$ is conjugation invariant because $A$ is real. We now distinguish two cases, depending on whether unboundedness of $A\acts V$ comes from the weights $\lambda$ or the nilpotent part $N_{\lambda}$.
\bigskip

\underline{Case 1}: \emph{There exists $\lambda_{0}\in \Lambda$ with $\lambda_{0}(A)\not\subseteq S^1$.} 
Let $e_{1}, \dots, e_{p}$ be a basis for $V_{\C, \lambda_{0}}$ in which $A_{|V_{\C, \lambda_{0}}}$ is represented by upper triangular matrices. Write $W'_{\C}=\Span_{\C}(e_{1}, \dots, e_{p-1})$, note $W'_{\C}$ is $A$-invariant. Denoting by $\sigma$ the complex conjugation map on $\C^d$,  set
 $$W_{\C}:=(W'_{\C}+\sigma(W'_{\C})) \bigoplus \oplus_{\lambda \in \Lambda\smallsetminus \{\lambda_{0}, \sigma\circ\lambda_{0}\}}V_{\C,\lambda}.$$
 Let $a\in A$ such that $|\lambda_{0}(a)|>1$. 
We check that every non-zero $u\in V_{\C}/W_{\C}$ satisfies $a^k u \rightarrow\infty,
 a^{-k} u\rightarrow 0$. If $\sigma \circ \lambda_{0}=\lambda_{0}$, then we can take $e_{1}, \dots, e_{p}$ in $V_{\C, \lambda_{0}}\cap V$, then $W'_{\C}=\sigma(W'_{\C})$, and $V_{\C}/W_{\C}\simeq \C e_{p}$ yielding the claim. 
 If $\sigma \circ \lambda_{0} \neq \lambda_{0}$, then $V_{\C,\lambda_{0}}$ and $\sigma V_{\C,\lambda_{0}}=V_{\C,\sigma\circ \lambda_{0}}$ are in direct sum. In this case $V_{\C}/W_{\C}\simeq \C e_{p}\oplus \C \sigma e_{p}$ and the claim follows as well. 
Finally, as $W_{\C}$ is $\sigma$-invariant (i.e. defined over $\R$), the escape property on $V_{\C}/W_{\C}$ descends to $V/W$. 
 
 \bigskip
\underline{Case 2}: \emph{Every $\lambda \in \Lambda$ satisfies $\lambda(A)\subseteq S^1$.} 
Let $\lambda_{1}\in \Lambda$ such that $A\acts V_{\C,\lambda_{1}}$ has an unbounded orbit.  The Lie algebra of $A_{|V_{\C,\lambda_{1}}}$ can be put in upper triangular form. Let $X\in \Lie(A)$. The assumption of case $2$ tells us $X_{|V_{\C,\lambda_{1}}}$ has diagonal coefficients in $i\R$ (they are purely imaginary). Unboundedness  guarantees we may choose $X_{|V_{\C,\lambda_{1}}}$ to have non-trivial nilpotent part: $X_{|V_{\C,\lambda_{1}}}= it \id_{V_{\C,\lambda_{1}}}+ N$ where $t\in \R$ and $N \in \End(V_{\C,\lambda_{1}})$ is nilpotent non-zero. Let $m\geq 1$ such that $N^m\neq0$, $N^{m+1}=0$. Set $a_{k}=e^{kX}$, $W_{\C}=\oplus_{\lambda \in \Lambda\smallsetminus \{\lambda_{1}, \sigma\circ\lambda_{1}\}}V_{\C,\lambda}$.  Set  $L_{\C}=\Ker  N^m +\sigma(\Ker N^m)\subseteq V_{\C}/W_{\C}$.
We check that every non-zero $u\in V_{\C}/W_{\C}\smallsetminus L_{\C}$ satisfies $a_{k}u\rightarrow\infty$ and $a_{k}^{-1}u \rightarrow\infty$. Assume $\sigma\circ \lambda_{1}=\lambda_{1}$. Then $V_{\C}/W_{\C}\simeq V_{\C,\lambda_{1}}$ and $L_{\C}=\Ker N^{m}$. For $u\in V_{\C,\lambda_{1}}$, we have 
$a_{k}u=e^{i kt }e^{k N}u=e^{ikt}(\sum_{j=0}^m  \frac{k^j}{j!}N^j u)$. If $u\notin \Ker N^{m}$, then up to a phase, this is a polynomial in $k$ of degree $m$, whence it escapes to infinity as $k\to\pm\infty$.  In the other case where $\sigma\circ \lambda_{1} \neq \lambda_{1}$, we have $V_{\C,\lambda_{1}}$ and $V_{\C,\sigma \circ \lambda_{1}}$ in direct sum, and $V_{\C}/W_{\C}\simeq V_{\C,\lambda_{1}} \oplus V_{\C,\sigma \circ \lambda_{1}}$. We can then argue similarly. 
Finally, as $W_{\C}$ and  $L_{\C}$ are $\sigma$-invariant (i.e. defined over $\R$), the escape property on $(V_{\C}/W_{\C})\smallsetminus L_{\C}$ descends to $(V/W)\smallsetminus L$. 
\end{proof}

We now deal with solvable Lie groups $G$ whose conjugation action on the quotient vector space of $\overline{[G,G]}$ has some unbounded orbit. A basic example  is given by the affine group from \Cref{ex-affinegroup}.

\begin{lemma}\label{solvable-SEP-2}
Let \(G\) be a connected solvable Lie group, $S=\overline{[G,G]}$. If the action $G\acts S/S'$ has an unbounded orbit, then a quotient of \(G\) satisfies \(\SEP\).
\end{lemma}

\begin{proof}
If $G$ is nilpotent, then the result follows by \Cref{nilp-SEP}. 
Therefore, we may assume $G$ non-nilpotent. 
Up to quotienting by $S'$, we may assume $S'=\{\id\}$, in particular the nilpotent subgroup $S\unlhd G$ can be seen as an embedded vector space. Noting the conjugation action $G\acts S$ factors through $G/S$ and $G/S$ is abelian, \Cref{abelian-linear-SSEP}  tells us that up to quotienting a second time, there exists a proper subspace $L\subseteq S$ and a sequence $(a_{k})_{k}\in G^{\N}$ such that every $s\in S\smallsetminus L$ satisfies 
$$\Ad(a_{k})s\to\infty\quad \quad \quad \Ad(a^{-1}_{k})s\to 0 \text{ or }\infty.$$

Let $m\geq 1$. We define an open dense set of $\ux\in G^m$ which will ultimately validate $\SEP$. Let $\omega\neq \omega'\in \Omega_{m}$.  Set
$$
O_{\omega,\omega'} := \left\{
    \begin{array}{ll}
     \{\ux\in G^m \,:\, \ux_{\omega, \omega'}\in S\smallsetminus L\} & \mbox{if } \omega_{ab} =\omega'_{ab} \\
       \{\ux\in G^m \,:\, \Ad(\ux_{\omega, \omega'})_{|S} \neq  \id_{S}\}  & \mbox{if  $\omega_{ab} \neq \omega'_{ab}$.}
\end{array}
\right.
$$
Let us check $O_{\omega,\omega'}$ is nonempty. In the case where $\omega_{ab}=\omega'_{ab}$, having  $O_{\omega,\omega'}=\emptyset$ would imply that every commutator $[a,b]$ ($a,b\in G$) is in $L$ (\Cref{Iww'-large}), which in turn would yield $L=S$, which is absurd.  In the  case where $\omega_{ab}\neq \omega'_{ab}$,  having $O_{\omega,\omega'}=\emptyset$ would imply that $\Ad(G)_{|S}=\{\id_{S}\}$, i.e. $G$ is nilpotent (see \Cref{solv->nil}), which we have excluded from the start. We have thus checked $O_{\omega,\omega'}$ is nonempty. By \Cref{avoidance}, each $O_{\omega,\omega'}$ is dense (and open), therefore the set $O:=\cap_{\omega\neq \omega'}O_{\omega,\omega'}$ is also open and dense in $G^m$.

 We now fix $\ux\in O$ and construct a switching sequence for $\ux$. Let $\omega\neq \omega'\in \Omega_{m}$ with  $\omega_{ab} =\omega'_{ab}$. The property that $\ux_{\omega, \omega'}\in S\smallsetminus L $ implies 
$$\Ad(a_{k}) \ux_{\omega, \omega'} \to\infty\quad \quad \quad \Ad(a^{-1}_{k})\ux_{\omega, \omega'} \to 0 \text{ or }\infty.$$
Moreover, projecting to $G/S$, we also see that for $\eps\in \{-1,1\}$, we have 
$$a^{\eps}_{k}\ux_{\omega,\omega'} a^\eps_{k}\to \infty.$$
 This justifies the escape property for pairs $\omega,\omega'$ such that   $\omega_{ab} =\omega'_{ab}$. 
 
 In order to deal with  the situations where  $\omega_{ab} \neq \omega'_{ab}$, we will replace the sequence $a_{k}$ by $a_{k}s_{k}$ where $s_{k}\in S$ is a suitable  parameter, going to infinity very fast.  Note such replacement does not affect the estimates obtained above in the case $\omega_{ab} = \omega'_{ab}$ because $\Ad(S)\acts S$ is trivial. 
 
 Let $\omega\neq \omega'\in \Omega_{m}$ with $\omega_{ab}\neq \omega'_{ab}$. 
By definition of $O$, we have
$$\|\Ad(\ux_{\omega, \omega'})_{|S} - \id_{S} \| \geq \eps$$
where $\eps>0$ depends only on $\ux$. 
This allows us to consider a sequence $(s_{k})_{k}\in S^{\N}$ and a constant $\eps'>0$, both depending only on $\ux$ (and not on $\omega,\omega'$),  such that in $S$ seen as a vector space, we have
\begin{equation*}
 \|\Ad(\ux_{\omega, \omega'})s_{k} - s_{k} \|  \geq \eps' \|s_{k}\| \to +\infty.
 \end{equation*}
Seen in $G$, this gives in particular
$$\Ad(s_{k})\ux_{\omega, \omega'}\to \infty$$
with a rate of escape related to the growth of $\|s_{k}\|$.
Up to choosing $\|s_{k}\|$ to grow fast enough in terms of $a_{k}$, we further have
$$\Ad(a_{k}s_{k})\ux_{\omega, \omega'}\to \infty. $$
Now, we also check expansion for the action of $(a_{k}s_{k})^{-1}$. 
As the conjugation action $G\acts S$ factors through $G/S$, and $G/S$ is abelian, we have for all $k\geq0$,  
\begin{equation} \label{div-ksab}
\|\Ad(s^{-1}_{k} a^{-1}_{k} \ux_{\omega} \ux_{\omega'}^{-1} a_{k}s_{k})s_{k} - s_{k} \|=  \|\Ad(\ux_{\omega, \omega'})s_{k} - s_{k} \|   \geq \eps' \|s_{k}\| \to_{k} +\infty. 
 \end{equation}
Note the first term can be seen in $G$ as
 $$s_{k} \Ad(s^{-1}_{k} a^{-1}_{k} \ux_{\omega, \omega'} a_{k}s_{k})(s^{-1}_{k})=\Ad(a^{-1}_{k})(\ux_{\omega, \omega'})(\Ad(s^{-1}_{k} a^{-1}_{k})(\ux_{\omega, \omega'}))^{-1}.$$ 
 Imposing $\|s_{k}\|$ to grow fast enough in terms of $\Ad(a^{-1}_{k})(\ux_{\omega, \omega'})$, Equation \eqref{div-ksab} forces 
 $$\Ad(s^{-1}_{k} a^{-1}_{k})(\ux_{\omega, \omega'})\to \infty.$$
 On the other hand, by projection to $G/S$, we clearly have for $\eps\in \{-1,1\}$
 $$ (a_{k}s_{k})^\eps\ux_{\omega, \omega'}(a_{k}s_{k})^\eps \to \infty. $$
This finishes the proof of $\SEP$, with switching sequence $a_{k}s_{k}$.
\end{proof}

We now conclude the proof of \Cref{solv-SEP}.

%\noindent\emph{Remark}
%Although this extra characterization was not used during the proof of \Cref{sol-case}, we may observe 

\begin{proof}[Proof of \Cref{solv-SEP}] 
Denote by $d)$ the additional criterion in \Cref{solv-SEP}. We only need to check $b) \implies d) \implies c)$. The first implication is clear. The second is the combination of Lemmas \ref{solvable-SEP1}, \ref{solvable-SEP-2}.
\end{proof}

\subsection{The case where $G$ is semisimple}
We show   non-compact semisimple Lie groups always satisfy the simultaneous escape property.

\begin{lemma} \label{semisimple-case}
Let $G$ be a non-compact semisimple connected Lie group. Then $G$ satisfies $\SEP$.
\end{lemma}

\begin{proof}
\bigskip
Note the center $Z(G)$ of $G$ is a discrete subgroup. Set $H=G/Z(G)$. If $H$ is compact, then using that $H$ has finite fundamental group, we must have $Z(G)$ finite, whence $G$ is compact, which is absurd. Therefore $H$ is non-compact. Note $H$ also has trivial center.  By projecting to $H$, we may assume $Z(G)=\{e\}$. 

Fix a Cartan  subspace $\ka\subseteq \kg$. Let $\Phi\subseteq\ka^*$ be the associated system of restricted roots, so that
\[
\kg=\kg_0\oplus\bigoplus_{\alpha\in\Phi}\kg_\alpha,
\qquad
\kg_\alpha
=
\{Y\in\kg:\ [X,Y]_{\kg}=\alpha(X)Y\text{ for every }X\in\ka\}.
\]
For $\alpha\in\Phi$, we write $\pi_{\alpha}$ the projection to $\kg_{\alpha}$ parallel to  $\kg_0\oplus\bigoplus_{\beta\in\Phi,\,\beta\neq \alpha}\kg_{\beta}$. Choose a system  of positive restricted roots $\Phi^+\subseteq\Phi$. Write $\alpha_{\max}\in \Phi^+$ the highest root. Choose a line $\R v_{+}\subseteq \kg_{\alpha_{\max}}$ in the highest weight subspace, and  a line $\R v_{-}\subseteq \kg_{-\alpha_{\max}}$ in the lowest weight subspace.

Let $m\geq 1$. For $\omega\neq \omega'\in \Omega_{m}$, write 
$$O_{\omega,\omega'} :=\{\ux \in G^m\,:\, \pi_{\pm \alpha_{\max}}\Ad(\ux_{\omega,\omega'}) v_{\pm}  \neq 0\}$$
(signs being independent, so there are four conditions on $\ux$). 
As  $G$ is semisimple, \Cref{Iww'-large} implies that $\{\ux_{\omega,\omega'}\,:\, \ux\in G^m\}$ contains a neighborhood of $\id$, so $O_{\omega,\omega'}\neq \emptyset$. 
\Cref{avoidance} then implies that $O_{\omega,\omega'}$ is a dense open subset of $G^m$. The set  $O=\cap_{\omega \neq \omega'} O_{\omega,\omega'}$ is also dense and open. 

Let $\ux\in O$. Fix $X\in \ka^{++}$ the open Weyl chamber associated to $\Phi^+$. Then 
$$\Ad(e^{\pm kX}\ux_{\omega,\omega'} e^{-kX})v_{-} \to \infty \quad \quad \quad \Ad(e^{\pm kX}\ux_{\omega,\omega'} e^{kX})v_{+} \to \infty. $$
Therefore $e^{\pm kX}\ux_{\omega,\omega'} e^{-kX}\to \infty$, so $\SEP$ holds for $\ux$ with switching sequence $\sigma_{k}=e^{kX}$. 
\end{proof}

\subsection{The general case}

We are now able to establish  \Cref{prop-sim-switch} without structural assumption on the Lie group $G$. We first deal with the case where $[G,G]$ acts non trivially on a vector space quotient of a normal Lie subgroup of $G$.

\begin{lemma} \label{case-S/S'-nontriv}
Let $G$ be a connected Lie group. Assume there exists a normal connected Lie subgroup  $S\subseteq G$ such that the representation $\Ad([G,G])\acts S/S'$ is non-trivial. 
Then $G$ satisfies $\SEP$.
\end{lemma}

\begin{proof}
Let $m\geq 2$. For $\omega\neq \omega'\in \Omega_{m}$, write 
$$O_{\omega,\omega'} :=\{\ux \in G^m\,:\, \Ad(\ux_{\omega, \omega'})_{|S/S'} \neq \pm \id_{S/S'} \}.$$
Lemmas \ref{avoidance}, \ref{Iww'-large} and the assumption that  $\Ad([G,G])\acts S/S'$ is non-trivial together imply that $O_{\omega,\omega'}$ is open and dense in $G^m$. The set  $O=\cap_{\omega \neq \omega'} O_{\omega,\omega'}$ is thus also open and dense. Let $\ux\in O$. Let $(v_{k})\in (S/S')^{\N}$ be a sequence such that for all $\omega\neq \omega'\in \Omega_{m}$, 
$$\|( \id_{S/S'}\pm \Ad(\ux_{\omega,\omega'})_{|S/S'})(v_{k})\|\to +\infty.$$
Lifting $v_{k}$ to $s_{k}\in S$, we find that for all $\eps, \eps'\in \{-1,1\}$, we obtain that the sequence
$(s^{\eps}_{k}\ux_{\omega,\omega'} s_{k}^{\eps'} \ux_{\omega,\omega'}^{-1} )_{k} $ 
leaves all compact sets of $S$. As the injection map $S\hookrightarrow G$ is proper and $\ux_{\omega,\omega'}$ is bounded (depending on $\ux$ only),  we get that for each 
$\omega\neq \omega'\in \Omega_{m}$, the sequence $s^{\eps}_{k}\ux_{\omega,\omega'}s^{\eps'}_{k}$ goes to infinity in $G$.
\end{proof}

To deal with situations where $[G,G]$ acts trivially on all vector space quotients $S/S'$, we need the next preparatory lemma. 

\begin{lemma} \label{KactsN}
Let $N$ be a simply connected nilpotent Lie group. 
\begin{itemize}
\item If  $F\subseteq N$ is a  subgroup such that $N=F [N,N]$ then $F=N$. 
\item If $K$ is a compact   Lie group acting on $N$ by automorphisms, and the quotient action $K\acts N/[N,N]$ is trivial, then $K\acts N$ is trivial.
\end{itemize}
\end{lemma}

%\noindent\emph{Remark}. It is important that the group be nilpotent. For a solvable group, the claim is false. For example, $S=\SO_{2}\ltimes \R^2$ satisfies $[S,S]=\R^2$, so $F=\SO_{2}$ has surjective projection to $S/[S,S]$, but $F\neq S$.  To contradict the second claim, consider the action of $\SO_{2}$ by conjugation on $S$. 

\begin{proof}

We may  identify $N$ with its Lie algebra via the exponential map (\Cref{fact-nilp-spc}). This way $N$ also has a vector space structure.

Let us check the first claim. 
Note we have $N=FN^{[2]}=F+N^{[2]}$ where the first product refers to the group structure, and the addition refers to the vector space structure. Let $s$ be the step of $N$. Note that every $s$-bracket $\ad(x_{1})\ad(x_{2}) \dots \ad(x_{s-1}) x_{s}$ is invariant under replacing any $x_{i}$ by $x_{i}+y$ where $y\in N^{[2]}$. By assumption, every $x\in N$ is of the form $x_{F}+y$ where $x_{F}\in F$, $y\in N^{[2]}$. As the $s$-bracket  can be expressed via intertwined commutators (for the group structure), we deduce that  $F$ contains $N^{[s]}$. We are thus reduced to $N/N^{[s]}$, whence the proof by induction. 

Let us deal with the second claim. Write $\dd k$ the Haar probability measure on $K$. For every $x\in N$, the average $\int_{K}k.x \dd k$ is  $K$-invariant, and it coincides with $x$ modulo $[N,N]$. Hence the group of $K$-fixed points $F:=\{y \in N\,:\, K.y=y\}$ satisfies $F  [N,N]=N$. It follows from the previous paragraph that $F=N$, i.e. the action $K\acts N$ is trivial. 

\end{proof}

We finally conclude the proof of \Cref{prop-sim-switch}.

\begin{proof}[Proof of \Cref{prop-sim-switch}]
We only need to check that $b) \implies c)$. 
Denote by $R$ the  radical of $G$. By \Cref{semisimple-case}, we may assume that $G/R$ is compact. Let $S=\overline{[G,G]}$, write $R_{S}$ the  radical of $S$. In this paragraph, we show that $S/R_{S}$ is compact. Observe $S/(R\cap S) \simeq [G/R, G/R]=G/R$ is compact. Moreover, denoting by $(R\cap S)^o$ the identity component of $R\cap S$, we have on the one hand that the covering $S/(R\cap S)^o\rightarrow  S/(R\cap S)$ has finite fibers (see Weyl's theorem from \Cref{Weyl}), so $S/(R\cap S)^o$ is compact, and on the other hand, $(R\cap S)^o\subseteq R_{S}$ by inspecting definitions. Hence $S/R_{S}$ is indeed compact. 
By \Cref{Levi-G/R-compact}, we can then write $S=K_{S}R_{S}$ where $K_{S}\subseteq S$ is a compact semisimple connected Lie subgroup with $K_{S}\cap R_{S}$ finite. 

Let $V$ be the largest vector space quotient of $\overline{[R_{S},R_{S}]}$.   If the representation $\Ad(R_{S})\acts V$ is non-trivial, then \Cref{case-S/S'-nontriv} implies $\SEP$, i.e. $c)$ holds. Therefore we may assume this representation trivial. By \Cref{solv->nil}, we then have that $R_{S}$ is nilpotent.

Let $W$ be the largest vector space quotient of $R_{S}$. If the representation $\Ad(S)\acts W$ is non-trivial, then \Cref{case-S/S'-nontriv} again  implies  $c)$. From now on we assume that $\Ad(S)\acts W$ is trivial.  Up to quotienting by its largest compact subgroup (which is normal in $G$), we may assume that $R_{S}$ is simply connected, see \Cref{max-comp-in-nilp}. The triviality of $\Ad(K_{S})\acts W$ then yields via \Cref{KactsN} that $\Ad(K_{S})\acts R_{S}$ is trivial. This means that $K_{S}$ and $R_{S}$ commute. In view of the uniqueness statement in \Cref{Levi-G/R-compact}, we deduce  that $K_{S}$ is uniquely defined in $S$, therefore normal in $G$. Now the Lie group $G/K_{S}$ satisfies that $[G/K_{S}, G/K_{S}]$ is solvable, whence $G/K_{S}$ is also solvable. As $G/K_{S}$ still satisfies the assumption $b)$,  \Cref{solv-SEP} dealing with the solvable case implies  $c)$ for $G/K_{S}$, whence for $G$. 

\end{proof}

 %%%%%%%%%%%%%%%%%%%%%%%%%%%%%%%%%%%%%%%%%%%%%%%%%%%%%%%%%%%%%%%%%%%%%%%%%%%%%%%%%%%%%%% %%%%%%%%%%%%%%%%%%%%%%%%%%%%%%%%%%%%%%%%%%%%%%%%%%%%%%%%%%%%%%%%%%%%%%%%%%%%%%%%%%%%%%% %%%%%%%%%%%%%%%%%%%%%%%%%%%%%%%%%%%%%%%%%%%%%%%%%%%%%%%%%%%%%%%%%%%%%%%%%%%%%%%%%%%%%%% %%%%%%%%%%%%%%%%%%%%%%%%%%%%%%%%%%%%%%%%%%%%%%%%%%%%%%%%%%%%%%%%%%%%%%%%%%%%%%%%%%%%%%% %%%%%%%%%%%%%%%%%%%%%%%%%%%%%%%%%%%%%%%%%%%%%%%%%%%%%%%%%%%%%%%%%%%%%%%%%%%%%%%%%%%%%%% %%%%%%%%%%%%%%%%%%%%%%%%%%%%%%%%%%%%%%%%%%%%%%%%%%%%%%%%%%%%%%%%%%%%%%%%%%%%%%%%%%%%%%% %%%%%%%%%%%%%%%%%%%%%%%%%%%%%%%%%%%%%%%%%%%%%%%%%%%%%%%%%%%%%%%%%%%%%%%%%%%%%%%%%%%%%%% %%%%%%%%%%%%%%%%%%%%%%%%%%%%%%%%%%%%%%%%%%%%%%%%%%%%%%%%%%%%%%%%%%%%%%%%%%%%%%%%%%%%%%% %%%%%%%%%%%%%%%%%%%%%%%%%%%%%%%%%%%%%%%%%%%%%%%%%%%%%%%%%%%%%%%%%%%%%%%%%%%%%%%%%%%%%%% %%%%%%%%%%%%%%%%%%%%%%%%%%%%%%%%%%%%%%%%%%%%%%%%%%%%%%%%%%%%%%%%%%%%%%%%%%%%%%%%%%%%%%% %%%%%%%%%%%%%%%%%%%%%%%%%%%%%%%%%%%%%%%%%%%%%%%%%%%%%%%%%%%%%%%%%%%%%%%%%%%%%%%%%%%%%%% %%%%%%%%%%%%%%%%%%%%%%%%%%%%%%%%%%%%%%%%%%%%%%%%%%%%%%%%%%%%%%%%%%%%%%%%%%%%%%%%%%%%%%%
\section{Necessary condition to be Choquet-Deny}\label{Sec-onlyif}

We prove the ``only if'' part of \Cref{main-thm}, namely:

\begin{proposition}\label{notCD-sufficient-Liegrp}
Let \(G\) be a connected Lie group such that there exists \(x\in[G,G]\)
with \(\Ad(G)x\) unbounded. Then there exists a symmetric adapted
probability measure \(\mu\) on \(G\) such that \((G,\mu)\) is not Liouville.
\end{proposition}

%
%\begin{proof}
%Let \(N=\ker\pi\). Replacing the measure on \(Q\) by its lazy version, we may
%suppose that it gives positive mass to the identity; this does not change its
%bounded harmonic functions. Denote the resulting symmetric adapted
%non-Liouville measure by \(\bar\mu\). Choose a Borel section \(\sigma:Q\to G\) with \(\sigma(\id)=\id\), and set \(\nu=\sigma_*\bar\mu\). Choose a symmetric probability measure \(\kappa\) on \(N\), with an atom at the identity, whose support generates a dense subgroup of \(N\), and put \[  \mu=\frac12\bigl(\kappa*\nu*\kappa+\kappa*\check\nu*\kappa\bigr), \] where \(\check\nu\) is the image of \(\nu\) under inversion. Then \(\mu\) is symmetric and \(\pi_*\mu=\bar\mu\). Since both \(\kappa\) and \(\nu\) have an atom at the identity, the support of \(\mu\) generates a dense subgroup of \(G\). Finally, the pullback by \(\pi\) of a non-constant bounded \(\bar\mu\)-harmonic function is a non-constant bounded \(\mu\)-harmonic function. \end{proof}

The next lemma claims that the triviality of all harmonic bounded functions implies   that random walks on $G$ eventually ``forget'' their starting point. 

\begin{lemma}\label{CD-vs-mixing}
Let $G$ be a Lie group, let $\mu$ be a probability measure on $G$ with $\mu(\id)>0$.  If the pair $(G, \mu)$ is Liouville then for all \emph{absolutely continuous} probability measures $\nu_{1}, \nu_{2}$ on $G$, we have
$$\|\nu_{1}*\mu^{*n} -\nu_{2}*\mu^{*n}\| \rightarrow 0$$
where $\|\cdot \|$ stands for the total variation norm.
\end{lemma}
\begin{proof}
This  is an application of Foguel's convolution
$0$--$2$ law \cite[Theorem~2]{Foguel75}.
For extensions and alternative proofs of $0$--$2$ laws in the
general framework of Markov processes, see Derriennic
\cite{Derriennic1976} and Kaimanovich
\cite{Kaimanovich1992}.
\end{proof} 

To show that such $\mu$ is not Liouville, it will thus be enough  to check that for some $h\in G$, the distributions $\delta_{h}*\mu^{*n}$ and $\mu^{*n}$ cannot coincide, even up to some small (but macroscopic) left perturbation. We fix a \emph{right-invariant} Riemannian metric on $G$,  we write $d(\cdot, \cdot)$ the induced distance function, and $B_{r}$ the ball of radius $r>0$ centered at $\id$. 

%Note that in the above, we may restrict to $\nu_{1}, \nu_{2}$ of the form $\nu_{i}=\Haar_{B_{r}}*\delta_{h_{i}}$ where $\Haar_{B_{r}}$ is the right-invariant Haar measure restricted to the ball of radius $r>0$ centered at $\id$, and normalized to have mass $1$. In this case $\nu_{i}*\mu^{*n} $ can be seen as a mollification of $\delta_{h_{i}}*\mu^{*n}$ at scale $r$. 

%

\begin{proof}[Proof of \Cref{notCD-sufficient-Liegrp}]
It is enough to construct a non-Liouville symmetric adapted measure on a \emph{quotient} of $G$ (see \Cref{lift-symmetric-witness} below for details on this reduction). Therefore, by \Cref{prop-sim-switch}, we may assume that \emph{$G$ satisfies $\SEP$}.

In view of \Cref{CD-vs-mixing}, it is sufficient to construct a symmetric adapted probability measure $\mu$ on $G$  such that, writing $\tmu=\frac{1}{2}(\delta_{\id}+\mu)$, we have for  some $h\in G$,
\begin{equation}\label{eq-separation}
\|\Haar_{B_{1/2}}*\delta_{h}*\tmu^{*n} -\Haar_{B_{1/2}}*\tmu^{*n}\| \not\rightarrow 0.
\end{equation}

We will construct $\mu$ as a fractal measure, encoded by a set of infinite words, each pointing at a certain element in $G$. 
Given such a word, the associated element in $G$ is the limit of  elements corresponding to finitary prefixes of the word. In this paragraph we define those prefixes, and how they correspond to a group element. 
Let us consider some parameters, given by sequences $(r_{k})\in (0, 1)^{\N^*}$ with $r_{1}=1/2$,  $(q_{k})\in \N^{\N^*}$  with $q_{k+1}>q_{k}$ and $q_{1}=1$, and $(\sigma_{k})\in G^{\N^*}$. Further conditions on them will be specified later. 
 For $l\in \N $, set 
$$\cW_{l}:=\{(k, \ui) \,:\, k\geq 1, \, \ui \in \{1, 2\}^{\llbracket q_{k}, l \rrbracket}\}$$
where $\llbracket q_{k}, l \rrbracket :=[q_{k}, l ] \cap \N$, and
in the case $l<q_{k}$, we interpret the condition $\ui \in \{1, 2\}^{\llbracket q_{k}, l \rrbracket}$ as $\ui=\emptyset$ (the empty word). 
In particular,  $\cW_{0}=\N^*\times \{\emptyset\}$.  For $l\geq 1$, set $\cW^*_{l}=\cW_{l} \smallsetminus \cW_{0}$. We then have $$\cW_{<\infty}:=\cup_{l\in \N} \cW_{l}= \cW_{0} \bigsqcup \sqcup_{l\geq 1} \cW^*_{l}.$$ 
We define  a collection of group elements $(g_{k, \ui})_{(k,\ui) \in \cW_{<\infty}}$. For subscripts in $\cW_{0}$, we set $g_{k, \emptyset}=\sigma_{k}$. We then define the  $g_{k, \ui}$'s indexed by each $\cW^*_{l}$ using induction on $l$.  Suppose the elements $g_{k, \ui}$ have been defined for all $(k, \ui)\in \cW^*_{l}$ (this condition is empty if $l=0$). 
Let $(k, \ui)\in \cW^*_{l+1}$. There exists $(k , \ui')\in \cW_{l}$ such that $(k , \ui)=(k , \ui' \eps)$ where $\eps\in \{1, 2\}$. We then choose $g_{k, \ui}\in B_{r_{l+1}}(g_{k, \ui' })$. As $G$ is assumed to satisfy $\SEP$, we may impose that this new collection $(g_{k, \ui})_{(k, \ui)\in \cW^*_{l+1}}$ satisfies $\SEP$, for some switching sequence with diverging squares (see remark after \Cref{def-SEP}). Note such choice allows to further restrict $g_{k, \ui'1}, g_{k, \ui'2}$ to any prescribed open subsets of $B_{r_{l+1}}(g_{k, \ui' })$. 
In particular, considering $k=1$, we may also impose that, for every $m\geq1$,  the set of multiplicative differences $\{g_{1,\ui'2}\,(g_{1,\ui'1})^{-1} \,:\, \ui'\in \{1,2\}^{\llbracket1, m\rrbracket}\}$ spans a subgroup of $G$ which is $\eta_{m}$-dense in the ball $B_{1}$, for some sequence $(\eta_{m})_{m}$ going to $0$ as $m$ goes to infinity. This second requirement will ensure the symmetric measure $\mu$ we construct has its support spanning a dense subgroup.

We now define the measure $\mu$. Set 
$$ \cW:=\{(k, \ui) \,:\, k\geq 1, \, \ui \in \{1, 2\}^{\N_{ \geq q_{k}}}\},$$
which we may see as the limit of $(\cW_{l})_{l}$, and
 set $\pi_{l}: \cW\rightarrow \cW_{l}$ the natural projection map (restricting the coordinate $\ui$ to the interval  $\llbracket q_{k},l\rrbracket$). 
Assuming $r_{l+1}\leq r_{l}/2$, it makes sense to define limit points $(g_{(k, \ui)})_{(k, \ui)\in \cW}$  by setting $g_{(k, \ui)}=\lim_{l} g_{\pi_{l}(k, \ui)}$. Indeed, for any $k, \ui$, the sequence $(g_{\pi_{l}(k, \ui)})_{l\geq q_{k}}$ is Cauchy.
We then put on $\cW$ the probability measure 
$$\bP= c\sum_{k\geq 1} k^{-5/4} \delta_{k}\otimes \left(\frac{\delta_{1}+\delta_{2}}{2}\right)^{\otimes \N_{ \geq q_{k}}}$$
 where $c>0$ is chosen to guarantee mass $1$. We let $\mu_{0}$ be the distribution of $g_{(k, \ui)}$ as $(k, \ui)$ varies with law $\bP$. Note that for every $m\geq1$, every $\ui'\in \{1,2\}^{\llbracket1, m\rrbracket}$, the support of $\mu_{0}$ contains elements at distance at most $2r_{m+2}$ from $g_{1,\ui'1}$, and similarly for $g_{1,\ui'2}$.  Provided $r_{m+2}$ is chosen small enough depending on $\{g_{1,\ui'2}\,(g_{1,\ui'1})^{-1} \,:\, \ui'\in \{1,2\}^{\llbracket1, m\rrbracket}\}$, the  almost-density requirement from the previous paragraph yields that $\supp \mu_{0}$ generates a dense subgroup.
Finally, we set 
 $$\mu=\frac{1}{2}(\mu_{0}+ \text{inv}_{\star}{\mu}_{0})$$
  where $\text{inv}:g\mapsto g^{-1}$ denotes the inversion map.

\bigskip

We now check \eqref{eq-separation} for some $h$, up to further adjusting the parameters $(r_{k}, q_{k}, \sigma_{k})$. Given $a\in \cW$, set $k_{a}\in \N^*$ its first coordinate, so $a\in \{k_{a}\}\times \{1, 2\}^{\N_{ \geq q_{k_{a}}}}$. For a string $\ua=(a_{p})_{p\geq 1}\in  \cW^{\N^*}$, write $M_{\ua, P}=\max_{1\leq p\leq P} k_{a_{p}}$. 
Let $P_{0}, M_{0}\geq 1$ be  parameters and set 
$E=E((q_{k}), P_{0}, M_{0})$ the set of instructions $\ua=(a_{p})_{p\geq 1}\in  \cW^{\N^*}$ satisfying
\begin{itemize}
\item $M_{\ua, P_{0}}\leq M_{0}$.
\item For every $P> P_{0}$,  the maximum $M_{\ua, P}$ of the tuple $(k_{a_{1}}, \dots, k_{a_{P}})$ appears exactly once and satisfies $M_{\ua, P}\geq P^2$; moreover the projections $\pi_{q_{M_{\ua, P}}-1}(a_{1}), \dots, \pi_{q_{M_{\ua, P}}-1}(a_{P})$ are pairwise distinct. 
\end{itemize}  
 By \cite[Lemma 2.6]{FHTVF},    we may choose absolute constants $(q_{k})$ and $P_{0}, M_{0}$ so that $\bP^{\otimes \N^*}(E)\geq 3/4$. From now on, the sequence $(q_{k})$ will not be  specified further, all conditions on it have already been imposed.
We claim that for suitable $h$ and $(r_{k}, \sigma_{k})_{k}$, for every $\ua, \ua'\in E$, every $s,t\in \N$, every collection of signs $\eps_{1}, \dots, \eps_{s}, \eps'_{1}, \dots, \eps'_{t}\in \{-1,1\}$ we have 
\begin{equation}\label{eq-separation-mn}
d(hg^{\eps_{1}}_{a_{1}} \dots g^{\eps_{s}}_{a_{s}}, \,g^{\eps'_{1}}_{a'_{1}} \dots g^{\eps'_{t}}_{a'_{t}}) >1.
\end{equation}
Note that \eqref{eq-separation-mn} implies  \eqref{eq-separation}, and therefore \Cref{notCD-sufficient-Liegrp}.

Assume by contradiction that \eqref{eq-separation-mn} fails, i.e.
\begin{equation}\label{eq-separation-mn-fails}
d(hg^{\eps_{1}}_{a_{1}} \dots g^{\eps_{s}}_{a_{s}}, \,g^{\eps'_{1}}_{a'_{1}} \dots g^{\eps'_{t}}_{a'_{t}})\leq 1.
\end{equation}
 Up to imposing $d(h,e)$ large enough in terms of  $P_{0}, M_{0}$ and $\sigma_{1}, \dots, \sigma_{M_{0}}$, we  have $\max(M_{\ua, s}, M_{\ua', t})> M_{0}$. Up to imposing that for every $k>M_{0}$, we have $d(\sigma_{k}, \id)$  large enough in terms of $h$, $(\sigma_{1}, \dots, \sigma_{k-1})$, we must in fact have $M_{\ua, s}=M_{\ua', t}> M_{0}$.  In particular  $\min(s,t)>P_{0}$.

Using $s>P_{0}$, we may write $g^{\eps_{1}}_{a_{1}} \dots g^{\eps_{s}}_{a_{s}}=xg^{\eps}y$ where $\eps\in \{-1,1\}$ and $g$ is  the switching element, more precisely $g=g_{a_{j}}$ where $k_{a_{j}}=M_{\ua, s}$, $\eps=\eps_{j}$, and $x$ (resp. $y$) is the prefix (resp. the suffix) of the product. Similarly, using $t>P_{0}$, we have a decomposition $g^{\eps'_{1}}_{a'_{1}} \dots g^{\eps'_{t}}_{a'_{t}}=x'{g'}^{\eps'}y'$. By right invariance of the distance,  \eqref{eq-separation-mn-fails} amounts to
\begin{equation}\label{eq-separation-mn-fails2}
d(hx,\, x' {g'}^{\eps'} y'y^{-1}g^{-\eps}) \leq 1.
\end{equation}

We now  remove the contribution of ${g'}^{\eps'} y'y^{-1}g^{-\eps}$ in \eqref{eq-separation-mn-fails2}.  For that we need more conditions on $(\sigma_{k})_{k>M_{0}}$ and $(r_{q_{k}})_{k>M_{0}}$. Fix $k>M_{0}$. Regarding $\sigma_{k}$, note that elements of $\cW^*_{q_{k}-1}$ have their first coordinate in $\llbracket 1, k-1\rrbracket$, in particular we may define $\sigma_{k}$ conditionally  to $q_{k}$, $(g_{a} )_{a\in \cW^*_{q_{k}-1}}$ and $h$.
Given a  tuple $\omega=((w_{1},\eps_{1}) \dots (w_{p}, \eps_{p}))$ where $p\geq 0$, the $w_{i}$ are distinct elements in $\cW^*_{q_{k}-1}$ and the $\eps_{i}$ are in $\{-1,1\}$, set  $z_{\omega}=g^{\eps_{1}}_{w_{1}}\dots g^{\eps_{p}}_{w_{p}}$. Write $R_{k}:=2d(\sigma_{k-1},\id)\sharp \cW^*_{q_{k}-1}$ and $c_{k}> 0$ such that the diameter of $\Ad(B_{R_{k}}) B_{c_{k}}$ is at most $2^{-k}$. Recalling the $\SEP$ condition\footnote{More precisely, we assumed $\SEP$ with a switching sequence having diverging squares. This extra condition is used here to justify the case $\omega=\omega'$ in \eqref{SEP-applied}.} on $(g_{a} )_{a\in \cW^*_{q_{k}-1}}$, we  can choose $\sigma_{k}$ such that uniformly over all such tuples  $\omega, \omega'$ (potentially equal),  we have 
\begin{equation}\label{SEP-applied}
d\!\left(
\sigma^{\pm 1}_k z_{\omega} z_{\omega'}^{-1}\sigma^{\pm 1}_k, \, \id
\right)
\in [0,c_{k})\cup(10 + d(h, \id) +2R_{k}, \, \infty).
\end{equation}
Knowing $\sigma_{k}, q_{k}$ and $(g_{a} )_{a\in \cW^*_{q_{k}-1}}$, we further  impose $r_{q_{k}}$ to be small enough so that \eqref{SEP-applied} remains true under  $2 r_{q_{k}}$-perturbation: for all tuples $\omega, \omega'$ as above,
for every  $\tz_{\omega}, \hz_{\omega'}, {\tilde \sigma}_{k}, {\widehat \sigma}_{k}$ satisfying  $\tz_{\omega}=h_{1}\dots h_{p}$ where $d(h_{i}, g^{\eps_{i}}_{w_{i}})<2r_{q_{k}}$, a similar property for $\hz_{\omega'}$, and $d({\tilde \sigma}_{k}, \sigma_{k}), d({\widehat \sigma}_{k}, \sigma_{k}) <2r_{q_{k}}$, we have 
\begin{equation}\label{eq-expansion-pert}
 d\!\left({\tilde \sigma}^{\pm 1}_{k} \tz_{\omega} {\hz_{\omega'}}^{-1} {\widehat \sigma}^{\pm 1}_{k}, \id
\right)
\in [0,c_{k})\cup(10 + d(h, \id) +2R_{k},\,\infty).
\end{equation}

 Applying \eqref{eq-expansion-pert}  with ${\tilde \sigma}^{\pm 1}_{k}=g'^{\eps'}$, $ \tz_{\omega}=y'$,  $\hz_{\omega'}=y$, ${\widehat \sigma}^{\pm 1}_{k}=g^{-\eps}$, we deduce from  \eqref{eq-separation-mn-fails2}  that the interval $(10 + d(h, \id) +2R_{k},\,\infty)$ from the above alternative is never met, therefore
$$d(x',\, x' {g'}^{\eps'} y'y^{-1}g^{-\eps}) \leq 2^{-M_{\ua, s}},$$
whence 
$$d(hx,\, x') \leq 1+2^{-M_{\ua, s}}.$$
Iterating the argument with the prefixes $hx, x'$, we end up with a contradiction. This finishes the proof of \eqref{eq-separation-mn}, whence establishes \Cref{notCD-sufficient-Liegrp}.

\end{proof}

\bigskip

The next lemma was used in the above proof.

\begin{lemma}\label{lift-symmetric-witness}
Let $G$ be a Lie group, let $Q$ be a quotient Lie group of $G$. 
If \(Q\) admits a symmetric adapted
non-Liouville probability measure, then so does \(G\).
\end{lemma}

\begin{proof}
Write $Q=G/H$ where $H$ is a closed normal subgroup of $G$. Consider such a measure $\mu_{Q}$ on $Q$. Let $\nu=\sigma_{\star}\mu_{Q}$ where $\sigma:Q\rightarrow G$ is a section, and  let $\check \nu$ be the image of $\nu$ under the inversion map $g\mapsto g^{-1}$. Let $\kappa$ be a fully supported symmetric probability measure on $H$. The desired measure on $G$ can be defined as $\mu_{G}=\frac{1}{4} (\nu + \check \nu) + \frac{1}{2} \kappa$. Details are left to the reader.
\end{proof}

We finish the paper by giving the proof of   \Cref{cpct-by-nilp-cor}, which further characterizes Choquet-Deny  connected Lie groups of the form $K\ltimes N$ with $K$  compact, $N$ nilpotent. 

\begin{proof}[Proof of \Cref{cpct-by-nilp-cor}]
The group $[G,G]$ is generated by
$H_{1}:=[K,K]$, $H_{2}:=[K,N]$, $H_{3}:=[N,N]$. If $[H_{i},N]$ is bounded for each $i$, then for every $h \in\cup_{i} H_{i}$, we get $\Ad(N)h$ bounded, and using that $K$ is compact, we conclude that $\Ad(G)h$ is bounded. By \Cref{main-thm}, this means $G$ is Choquet-Deny, so we have justified the converse direction in the corollary. 

We now assume $G$ to be Choquet-Deny, and show $[H_{i},N]$ is bounded for each $i$. To do so, we may always quotient by $T_{N}$ the largest torus in $N$, so that $N$ becomes simply connected (\Cref{max-comp-in-nilp}). By \Cref{main-thm}, it is sufficient to check that any $h\in G$ such that $\Ad(N)h$ is bounded must satisfy $[h,N]=\{\id\}$, i.e. must commute with $N$. To check this property, consider such an $h$, and let $v\in \kn$.  Note that $t\mapsto e^{t v}he^{-tv}h^{-1}$ ($t\in \R$) is  polynomial once viewed as taking values in the Lie algebra $\kn$ of $N$, and bounded in view of our choice of $h$. It must therefore be constant, and taking the derivative at $t=0$ gives $v=\Ad(h)v$. Hence $\Ad(h)$ acts trivially on $\kn$, meaning $h$ commutes with $N$. This finishes the proof.

\end{proof}
\bigskip
\noindent\textbf{Competing interests.} The authors declare that they have no competing interests.

\bibliographystyle{abbrv} %alpha-fr abbrv-fr plain-fr acm amsplain aomplain
\bibliography{Choquet-Deny-Lie-Groups}

\end{document}